\documentclass[12pt]{amsart}
\usepackage{amscd,amssymb}
\usepackage{color}
\usepackage[arrow,dvips,matrix,arrow,ps,color,line,curve,frame]{xy}
\usepackage{pgf, tikz}
\usepackage{url}

\makeatletter
\def\url@leostyle{%
  \@ifundefined{selectfont}{\def\UrlFont{\sf}}{\def\UrlFont{\small\ttfamily}}}
\makeatother
\usetikzlibrary{matrix,arrows}

\let\oldlabel=\label

\def\prellabel{\marginparsep=1em
    \def\label##1{\oldlabel{##1}\ifmmode\else\ifinner\else
         \marginpar{{\footnotesize\ \\ \tt
                    ##1}}\fi\fi}}

\def\codim{\operatorname{codim}}
\def\rank{\operatorname{rank}}

\def\pc{\operatorname{pc}}

\def\Ker{{\operatorname{Ker}}}

\def\cone{{\operatorname{cone}}}

\def\o{\operatorname{o}}

\def\int{{\operatorname{int}}}
\def\conv{\operatorname{conv}}

\def\aff{\operatorname{aff}}
\def\lin{\operatorname{lin}}
\def\vertex{\operatorname{vert}}
\def\rank{\operatorname{rank}}
\def\codim{\operatorname{codim}}
\def\Im{\operatorname{Im}}
\def\Hom{\operatorname{Hom}}

\def\Pol{\operatorname{Pol}}
\def\Cones{{\operatorname{Cones}}}

\def\join{\operatorname{join}}

\def\surj{\operatorname{surj}}
\def\inj{\operatorname{inj}}

\def\RR{{\mathbb R}}
\def\CC{{\mathbb C}}

\def\ZZ{{\mathbb Z}}
\def\NN{{\mathbb N}}

\def\FF{{\mathbb F}}

\def\xx{{\mathbf x}}

\def\Coker{{\operatorname{Coker}}}

\def\Pol{\operatorname{Pol}}

\let\epsilon=\varepsilon
\let\phi=\varphi
\let\theta=\vartheta

\advance\headheight1.15pt

\newtheorem{lemma}{Lemma}[section]
\newtheorem{corollary}[lemma]{Corollary}
\newtheorem{theorem}[lemma]{Theorem}
\newtheorem{proposition}[lemma]{Proposition}
\newtheorem{conjecture}[lemma]{Conjecture}

\theoremstyle{definition}
\newtheorem{definition}[lemma]{Definition}
\newtheorem{remark}[lemma]{Remark}
\newtheorem{example}[lemma]{Example}

\newtheorem{observation}[lemma]{Observation}

\begin{document}

\title[Hom-polytopes]{Hom-polytopes}

\author[T. Bogart]{Tristram Bogart$^\star$}
\address{Departamento de Matem\'aticas \\
Universidad de los Andes \\
Cra 1 No. 18A-10, Edificio H \\
Bogot\'a, 111711 \\
Colombia}
\email{tc.bogart22@uniandes.edu.co}

\author[M. Contois]{Mark Contois}
\address{Research in Motion \\
295 Phillip Street \\
Waterloo, Ontario \\
Canada N2L 3W8}
\email{mcontois@rim.com}

\author[J. Gubeladze]{Joseph Gubeladze$^\dag$}
\address{Department of Mathematics\\
         San Francisco State University\\
         1600 Holloway Ave.\\
         San Francisco, CA 94132, USA}
\email{soso@sfsu.edu}

\thanks{$^\star$Supported by NSF grant DMS-0441170}
\thanks{$^\dag$Supported by NSF grant DMS-1000641}


\maketitle

\begin{abstract}
We study the polytopes of affine maps between two polytopes -- the
\emph{hom-polytopes}. The hom-polytope functor has a left adjoint --
\emph{tensor product polytopes}. The analogy with the category of
vector spaces is limited, as we illustrate by a series of explicit examples
exhibiting various extremal properties. The main challenge for
hom-polytopes is to determine their vertices.  A polytopal analogue of
the rank-nullity theorem amounts to understanding how the vertex maps behave relative to their surjective and
injective factors. This leads to interesting classes of surjective maps. In the last two
sections we focus on two opposite extremal cases -- when the source
and target polytopes are both polygons and are either generic or regular.
\end{abstract}

\section{introduction}\label{intro}

\subsection{Motivation} The convex polytopes and their affine maps
form a natural habitat for a major part of the contemporary combinatorics (combinatorial commutative
algebra, toric algebraic geometry, tropical geometry, Ehrhart theory, linear and integer programming).
The category of polytopes is also an object of
study in its own right. The importance of such an approach is highlighted on the
last pages of \cite{ZiPOL}. By analogy with algebraic structures and
topological spaces, one could ask whether polytopes are also amenable
to a unifying analysis, which would provide a general context for
various important constructions and results. In the case of algebra
and topology such a unifying machine is homology theory. Since convex
polytopes are just one step away from the classical linear world, the
question can be put in very concrete terms: what are the polytopal
versions of $\Hom$, $\otimes$, $\Ker$, $\Coker$, $\text{Ext}$? How do
they fit into the current trends in polytope theory? Is there a more
universal (algebraic?) mechanism for addressing concrete challenges in
polytopes than, say, triangulations, analogous to the triangulations vs. homology dichotomy for topological spaces?

A notable example that suggests that these are natural questions is
the Billera-Sturmfels concept of fiber polytopes
\cite{BiStuFIBER}. This construction is expected to be the right kernel object in the category of polytopes. But the analogy with $\ker$ is yet to be fully explained and, more importantly, pushed further to include still conjectural co-kernel
objects in the same category.

In this paper we undertake the first step in the direction of categorial analysis of convex polytopes: we study the sets of affine maps
between two given polytopes, the \emph{hom-polytopes}. Curiously enough, apart from the motivation above, hom-polytopes show some relevance in quantum physics \cite{Physics}.

\subsection{Results.} That the hom-polytopes are in fact polytopes is a folklore fact; see Section \ref{hom}. The full blown analogy with vector spaces is a symmetric closed monoidal structure on the category of polytopes over which the category itself is enriched (Corollary \ref{enriched}). In particular, there is a natural \emph{tensor product} of polytopes, satisfying the usual conjunction with hom-polytopes. But, unlike the linear situation, the tensor product of polytopes exhibits interesting extremal properties (Example \ref{neighboring}). The material up to Corollary \ref{enriched} is modeled on vector spaces and the arguments are mostly skipped. The summary is given for the sake of
completeness. After posting the preprint on arXiv, we learned about Valby's undergraduate thesis \cite{Valby} which gives a very detailed treatment of the same material.

The description of the facets of hom-polytopes is very simple: a facet
consists of the maps mapping a chosen vertex of the source into a
chosen facet of the target. It is, therefore, the determination of \emph{vertices} which amounts to understanding the geometric consequences of our categorial-polytopal endeavor. Very rarely can one hope for a full description of the vertices of hom-polytopes in terms of the source and target polytopes. In Section \ref{vertices} we introduce several tractable classes of surjective vertex maps in arbitrary dimension (\emph{deflations,} \emph{face collapses}). Their analysis (Theorems \ref{vertexcomposite} and \ref{fGamma}), in particular, yield complete description of the rank 1 vertex maps, i.~e., the vertex maps whose images in the target polytope are segments (Corollary \ref{rank1}).

In Section \ref{examples} we present several examples of interesting vertex maps, making clear: (i) the limitation of the analogy between the categories of polytopes and vector spaces (\emph{vertex factorization of non-vertex maps, gaps in ranks}), and (ii) the distinction between the classes of surjective vertex maps, introduced in Section \ref{vertices}.


In Section \ref{generic} we are able to show that in the hom-polytope
between two generic polygons, all but a few vertices are simple
(Theorem \ref{thm:generic}), and we completely describe the
exceptions. (To this end we must first introduce an
appropriate algebraic parametrization of the set of pairs of polygons.) The
result is shown in two steps: (i) introducing combinatorial structures that reduce the problem to the claim that certain explicit multivariate polynomial determinants of 31 different types are non-degenerate, and then (ii) verifying the claim by effective methods, with use of the computer algebra system \texttt{Macaulay 2}.

Regular polygons give rise to interesting arithmetic functions -- the number of vertices of the hom-polytope $\Hom(P_n,P_m)$ between regular $n$- and $m$-gons. In Section \ref{regular}, by explicit polygonal constructions, we obtain the explicit full lists of the vertices of $\Hom(P_m,P_n)$ when $\min(m,n)\le4$. At the end of Section \ref{regular} we present computational results, based on \textsf{Polymake}, for the number of vertices of $\Hom(P_m,P_n)$ with $m,n\le8$.

\subsection{Affine geometry} In this and next two subsections we fix terminology and collect several general facts.

For the unexplained background material on convex polyhedral geometry the reader is referred to \cite[Ch.1]{KRIPO} and \cite[Ch.1,2]{ZiPOL}.

We will work exclusively in finite dimensional real vector spaces. An
\emph{affine subspace} of a vector space is the sum of a linear
subspace and a vector. For a vector space $E$ and a subset $X\subset
E$, the \emph{affine hull} $\aff(X)$ is the minimal affine subspace of
$E$ containing $X$, and the \emph{linear hull} $\lin(X)$ is the
parallel translate of $\aff(X)$ containing $0$:
\begin{align*}
\aff(X)=\{\lambda_1x_1+\cdots+\lambda_nx_n\ |\ &n\in\NN,\ x_1,\ldots,x_n\in X,\\
&\lambda_1,\ldots,\lambda_n\in\RR,\ \lambda_1+\cdots+\lambda_n=1\};\\
\\
\lin(X)=\{\mu_1(x_1-x_0)+\cdots+\mu_n(x_n-x_0)\ |\ &n\in\NN,\ x_0,x_1,\ldots,x_n\in X,\\
&\mu_1,\ldots,\mu_n\in\RR\}.
\end{align*}

An \emph{affine map} $f:E\to E'$ between two vector spaces is the
composition of a linear map $E\to E'$ with a parallel translation $E'\to E'$. A map $f:E\to E'$ is affine if and only if it respects barycentric coordinates:
$$
f\left(\sum_{i=1}^n\lambda_ix_i\right)=\sum_{i=1}^n\lambda_if(x_i)
$$
for all $n\in\NN$, $x_1,\ldots,x_n\in E$, and $\lambda_1,\ldots,\lambda_n\in\RR$ with $\sum_{i=1}^n\lambda_i=1$. More generally, let $A\subset E$ and $A'\subset E'$ be affine subspaces. A map $f:A\to A'$ is \emph{affine} if it is a restriction of an affine map between the ambient vector spaces.  The set of affine maps between $E$ and $E'$, denoted by $\aff(E,E')$, is a vector space in a natural way. The set of affine maps $A\to A'$, denoted by $\aff(A,A')$, becomes an affine subspace of $\aff(E,E')$ upon choosing an affine projection $\pi:E\to A$ with $\pi^2=\pi$ and applying the embedding:
$$
\aff(A,A')\to\aff(E,E'),\quad f\mapsto\iota\circ f\circ\pi,
$$
where $\iota:A'\to E'$ is the identity embedding. The resulting affine
structures on $\aff(A,A')$ for various $\pi$ are all isomorphic. As a
result, there is a well-defined notion of convexity in
$\aff(A,A')$. This space satisfies $\dim\aff(A,A')=\dim A\dim A'+\dim A'$.

For a subset of an affine space $X\subset A$, its convex hull will be denoted by $\conv(X)$. For $X\subset E$, $E$ a vector space, the \emph{conical hull} of $X$ will be denoted by $\cone(X)$:
$$
\cone(X)=\left\{\RR_+x_1+\cdots+\RR_+x_n\ |\ n\in\NN,\ x_1,\ldots,x_n\in X\right\},
$$
where $\RR_+$ refers to the set of nonnegative reals.

For a convex subset $X\subset A$, the \emph{relative interior} of $X$ in $\aff(X)\subset A$ will be denoted by $\int(X)$. The boundary of $X$ is defined by $\partial X=X\setminus\int(X)$.

For two convex sets $X$ and $Y$, $\hom(X,Y)$ denotes the set of affine
maps $X\to Y$ and $\aff(X,Y)$ denotes the set of affine maps
$\aff(X)\to\aff(Y)$. The natural embedding
$\hom(X,Y)\hookrightarrow\aff(X,Y)$ makes $\hom(X,Y)$ into a convex
subset of $\aff(X,Y)$ that satisfies:
\begin{enumerate}
\item
$\aff(X,Y)=\aff(\hom(X,Y))=\hom(X,\aff(Y))$,
\item
$\int(\hom(X,Y))=\hom(X,\int(Y))$.
\end{enumerate}

\subsection{Polytopes and cones}\label{polytopesandcones}
We only consider convex polytopes, i.~e., our polytopes are the
compact intersections of finitely many affine half-spaces, or
equivalently the convex hulls of finitely many points.

For a polytope $P$, the sets of its facets and vertices will be denoted by $\FF(P)$ and $\vertex(P)$, respectively.

For two polytopes in their ambient vector spaces $P\subset E$ and $Q\subset E'$, their \emph{join} is defined by
\begin{align*}
\join(P,Q)=\conv\{(x,0,0),\ (0,1,y)\ |\ x\in P,\ y\in Q\}\subset E\oplus\RR\oplus E'.
\end{align*}
Let $\iota_P$ and $\iota_Q$ be the obvious embeddings of $P$ and $Q$ into $\join(P,Q)$.
Every point $z\in\join(P,Q)$ has a unique representation $z=\lambda\iota_P(x)+(1-\lambda)\iota_Q(y)$, $\lambda\in[0,1]$. For two affine maps $f:P\to R$ and $g:Q\to R$ we have the affine map:
\begin{align*}
\join(P,Q)\to R,\quad\lambda\iota_P(x)+(1-\lambda)\iota_Q(y)\mapsto\lambda f(x)+&(1-\lambda) g(y),\\
&\lambda\in[0,1],\ x\in P,\ y\in Q.
\end{align*}
It is uniquely determined by its restriction to $\Im\iota_P$ and $\Im\iota_Q$.

By \emph{cones} we will refer to pointed, convex, polyhedral cones:
those that are obtained as the intersection of finitely many half-spaces and
contain no lines.

\medskip\noindent\emph{Further notation.} The \emph{bipyramid} over a polytope $P$ in a vector space $E$ is the polytope
$$
\Diamond(P)=\conv((P,0),c_P+(0,1),c_P-(0,1))\subset E\oplus\RR,
$$
where $c_p\in P$ is the barycenter.

The $n$-dimensional standard simplex, cube, and cross-polytope are defined by
\begin{align*}
&\Delta_n=\conv(e_1,\ldots,e_n,e_{n+1}),\\
&\Box_n=\conv\bigg(\sum_{i=1}^n\delta_ie_i,\, \delta_i=\pm1\bigg),\\
&\Diamond_n=\conv(\pm e_1,\ldots,\pm e_n),\\
\end{align*}
where $e_i$ denotes the $i$th standard basis vector.

For a natural number $n\ge3$, the standard regular $n$-gon is
\begin{align*}
P_n=&\conv(1,\zeta_n,\zeta_n^2,\ldots,\zeta_n^{n-1})\subset\CC=\RR\oplus\RR,\\
&\zeta_n=\cos(2\pi/n)+\sin(2\pi/n)i.
\end{align*}

\subsection{Categories} Our category theory terminology follows the classical source
\cite{macCAT}: comma categories, limits and co-limits of diagrams,
conjugated functors, and symmetric monoidal categories. For the
concept of enriched categories, we refer the
reader to \cite{Enriched}. Even if this is the first time the reader
encounters this terminology, the exposition is sufficiently
self-explanatory to warrant skipping inclusion of the definitions. However, the interested reader can consult \cite{Valby}.

Let $\Pol$ denote the category of polytopes and affine maps and $\Cones$ denote the category of cones and linear maps. In both categories, for objects $A$ and $B$ we will use the notation $\Hom(A,B)$ for the corresponding hom-sets.

In $\Pol$ we have the following universal equalities:
\begin{align*}
&\join(P,Q)=P\coprod Q=\lim_{\to}(P,Q),\\
&P\times Q=P\prod Q=\lim_{\leftarrow}\big(P,Q\big).\\
\end{align*}

\section{Hom-polytopes}\label{hom}
The following proposition is well-known \cite[\S9.4]{ZiPOL} and is the basis for the
special module in \textsf{Polymake} \cite{POLYMAKE,GaJoPOLY} for computing
hom-polytopes.\footnote{Called \emph{mapping polytopes} in \textsf{Polymake}.} We used this software for the experiments
presented in Section~\ref{regular} below. The details are written up to ease the references in the following sections.

\begin{proposition}\label{homproperties} Let $P,Q,R$ be polytopes.
\begin{enumerate}
\item $\Hom(P,Q)$ is a polytope in $\aff(P,Q)$ with
\begin{align*}
\FF(\Hom(P,Q))=\big\{H(v,F)\ |\ &v\in\vertex(P),\ F\in\FF(Q),\\
&H(v,F)=\{f\in\Hom(P,Q)\ |\ f(v)\in F\}\big\},
\end{align*}
\item $\dim\Hom(P,Q)=\dim P\dim Q+\dim Q$.
\item $\Hom(P,Q\times R)\cong\Hom(P,Q)\times\Hom(P,R)$,
\item $\Hom(P,Q\cap R)\cong\Hom(P,Q)\cap\Hom(P,R)$
\item $\Hom(\join(P,Q),R)\cong\Hom(P,R)\times\Hom(Q,R)$.
\end{enumerate}
\end{proposition}

\begin{proof}
(1) For every facet $F\subset Q$ we fix a surjective affine map
$\phi_F:Q\to\RR_+$, vanishing only on $F$.

\medskip\noindent\emph{Claim}. The system of affine maps
\begin{align*}
\phi_{F,v}:\aff(P,Q)\to\RR,\quad f\mapsto(\phi_F\circ f)(v),\quad F\in\FF(Q),\quad v\in\vertex(P),
\end{align*}
defines the facets of $\Hom(P,Q)$.

\medskip The equality
$$
\Hom(P,Q)=\bigcap_{F,v}\{f\in\aff(P,Q)\ |\ \phi_{F,v}(f)\ge0\}
$$
is straightforward. But for any vertex $v\in P$ and any facet
$F\subset Q$, there is an affine map $f:P\to Q$ such that
$f(v)\in\int(F)$ and $f(w)\in\int(Q)$ for each vertex $w\in P$ with
$w\not=v$. To obtain such map, take the composition of a parallel
projection of $P$ onto $[0,1]$, mapping $v$ to $0$ and the rest of $P$ onto $(0,1]$, with an embedding $[0,1]\to Q$, mapping $0$ to $\int(F)$ and $(0,1]$ to $\int(Q)$. Thus,
$$
f\in\phi_{F,v}(0)^{-1}\setminus\bigcup_{
{\tiny
\begin{matrix}
\FF(Q)\setminus\{F\}\\
w\in\vertex(P)\setminus\{v\}
\end{matrix}
}}
\phi_{G,w}(0)^{-1},
$$
which proves the claim.

(2) In the notation introduced above, we have
$$
\int(\Hom(P,Q))=\bigcap_{F,v}\{f\in\aff(P,Q)\ |\ \phi_{F,v}(f)>0\},
$$
a nonempty bounded open subset of $\aff(P,Q)$ defined by linear inequalities, so $\Hom(P,Q)\subset\aff(P,Q)$ is a full-dimensional polytope.

\medskip\noindent The parts (3,4,5) follow from the universal equalities in Section \ref{polytopesandcones} and the natural bijections of sets for any object $a$ and any diagram $\mathcal D$ in any category:
\begin{align*}
\Hom(a,\lim_{\leftarrow}\mathcal D)\cong\lim_{\leftarrow}\Hom(a,\mathcal D),\\
\Hom(\lim_{\to}\mathcal D,a)\cong\lim_{\leftarrow}\Hom(\mathcal D,a);
\end{align*}
these bijections are affine maps in our polytopal setting.
\end{proof}

For cones we have the following analogous statment:

\begin{proposition}\label{coneversion}
Let $C_1$, $C_2$, $C_3$ be cones.
\begin{enumerate}
\item
$\Hom(C_1,C_2)$ is a cone in $\Hom(\lin C_1,\lin C_2)$, whose facets are naturally indexed by the pairs $(R,F)$, $R\subset C_1$ an extremal ray and $F\subset C_2$ a facet.
\item $\Hom(C_1,C_2\times C_3)\cong\Hom(C_1,C_2)\times\Hom(C_1,C_3)$.
\item $\Hom(C_1\times C_2,C_3)\cong\Hom(C_1,C_3)\times\Hom(C_2,C_3)$.
\item $\Hom(C_1,C_2\cap C_3)\cong\Hom(C_1,C_2)\cap\Hom(C_1,C_3)$.
\end{enumerate}
\end{proposition}

\section{Homogenization, duals, and tensor product}\label{tensor}
To a polytope $P$ in a vector space $E$ we associate the \emph{homogenization cone}:
$$
C(P)=\cone(P,1),\ \text{where}\ (P,1)=\{(x,1)\ |\ x\in P\}\subset E\oplus\RR.
$$

\medskip A \emph{graded cone} is a cone $C$ together with an affine
map $\phi:C\to\RR_+$, satisfying the condition
$\phi^{-1}(0)=\{0\}$. The map $\phi$ is called a \emph{grading}. All
of our cones admit a grading \cite[Prop.1.21]{KRIPO}. For a graded cone $(C,\phi)$ we have the \emph{de-homogenization polytope} $\phi^{-1}(1)$, denoted by $C_{[1]}$.

For a full-dimensional polytope $P$ with $0\in\int(P)$ and a full
dimensional cone $C$ in a vector space $E$, the corresponding dual
objects in the dual space $E^{\o}$ are defined by
\begin{align*}
&P^{\o}=\{h\in E^{\o}\ |\ h(x)\le1\ \text{for all}\ x\in P\},\\
&C^{\o}=\{h\in E^{\o}\ |\ h(C)\subset\RR_+\}=\Hom(C,\RR_+).
\end{align*}
The duals are also full-dimensional in their ambient space $E^{\o}$.

For the general facts of dual polytopes and dual cones the reader is referred to \cite[\S1.B]{KRIPO} and \cite[\S2.3]{ZiPOL}. (The latter uses `polar' instead of `dual').

Unless specified otherwise, for a polytope $P$ in a vector space $E$, we consider the cone $C(P)$ as a graded cone w.r.t. the grading
$\xymatrix{C(P)\ar@{^{(}->}[r]&E\oplus\RR\ar[rr]^{\text{pr}_{\RR}}&&\RR}$.

For two polytopes $P$ and $Q$, let $\phi$ and $\psi$ denote the gradings $C(P),C(Q)\to\RR_+$, respectively. Then we have the grading on the product of the homogenization cones:
$$
C(P)\times C(Q)\to\RR_+,\quad(x,y)\to\phi(x)+\psi(y),
$$
which results in the isomorphism of polytopes
$$
(C(P)\times C(Q))_{[1]}\cong\join(P,Q).
$$

Further, if $\dim P=\dim E$ and $0\in\int(P)$ (i.~e., $(0,1)\in\int(P,1)$), then $C(P)^{\o}$ will be viewed as a graded cone via the grading $C(P)^{\o}\to\RR_+$, $h\mapsto h((0,1))$. It is crucial that $h((0,1))=0$ iff $h=0$.

\begin{definition}\label{tensordefinition}
Let $E$ and $E'$ be vector spaces, $C\subset E$ and $C'\subset E'$ be cones, and $P\subset E$ and $Q\subset E'$ be polytopes. We define the tensor products as follows:
\begin{align*}
&C\otimes C'=\cone\{(x\otimes y)\ |\ x\in C,\ y\in C'\}\subset E\otimes E',\\
&P\otimes Q=\conv\big(\big\{(x\otimes y, x, y)\ |\ x\in P,\ y\in Q\big\}\big)\subset\big(E\otimes E'\big)\oplus E\oplus E'.
\end{align*}
\end{definition}

\begin{proposition}\label{conetensor}
Let $C,C_1,C_2,C_3$ be cones and $P$ be a polytope.
\begin{enumerate}
\item
The bilinear map $C_1\times C_2\longrightarrow C_1\otimes C_2$, $(x,y)\mapsto x\otimes y$, solves the following universal problem: any bilinear map $C_1\times C_2\longrightarrow C_3$ passes through a unique linear map $\phi$,
$$
\xymatrix{
C_1\times C_2\ar@{=>}[r]\ar@{=>}[rd]_f&C_1\otimes C_2\ar@{.>}[d]_{\circlearrowright\ \ }^{\exists!\phi}\\
&C_3
}\ .
$$
Equivalently, $\Hom(C_1\otimes C_2,C_3)\cong\Hom(C_1,\Hom(C_2,C_3))$, i.~e., $\otimes,\Hom:\Cones\times\Cones\to\Cones$ form a pair of left and right adjoint functors.
\item $\Hom(C_1,C_2^{\o})^{\o}\cong C_1\otimes C_2\cong\Hom(C_2,C_1^{\o})^{\o}$ for $C_1$ and $C_2$ full-dimensional.
\item\begin{align*}
&\dim(C_1\otimes C_2)=\dim C_1\dim C_2=\dim\Hom(C_1,C_2),\\
&C\otimes\RR_+\cong C,\ (C_1\otimes C_2)\otimes C_3\cong C_1\otimes(C_2\otimes C_3),\ C_1\otimes C_2\cong C_2\otimes C_1,\\
&C_1\otimes(C_2\times C_3)\cong(C_1\otimes C_2)\times(C_1\otimes C_3).
\end{align*}

\item The extremal rays of $C_1\otimes C_2$ are the tensor products of the extremal rays of $C_1$ and $C_2$.
\item $\Hom(P,C)\cong\Hom(C(P),C)$, where the $\Hom$ on the right refers to the set of affine maps from $P$ to $C$.
\item If $P$ is full-dimensional and $0\in\int(P)$ then:
\begin{align*}
&C(P^{\o})\cong C(P)^{\o}\ \text{as graded cones, i.~e.,}\ P^{\o}\cong C(P)^{\o}_{[1]},\ \text{and}\\
&\Hom(P,[0,1])\cong C(P^{\o})\cap((0,1)-C(P^{\o})).
\end{align*}
\end{enumerate}
\end{proposition}

\begin{proof}
(1--3) are straightforward analogues of the corresponding linear algebra facts.

\medskip\noindent\emph{Notice.} Actually, (2) holds true for general, not necessarily full
dimensional, cones. This is so because $C^{\o\o}=C$ for general cones. Here we have to restrict to the full dimensional
case because from the beginning we have restricted to the case of pointed cones, and the dual of a cone is pointed iff the cone is full dimensional.

\medskip\noindent(4) Let $R_i\subset C_i$ be extremal rays, $\nu_i:C_i\to\RR_+$ be linear maps with $\nu_i^{-1}(0)=R_i$, and $\phi_i:C_i\to\RR_+$ be gradings, $i=1,2$. Then the linear map $$
\nu_1\otimes\phi_2+\phi_1\otimes\nu_2:C_1\otimes C_2\to\RR_+$$
satisfies the condition
$$
(\nu_1\otimes\phi_2+\phi_1\otimes\nu_2)^{-1}(0)=R_1\otimes R_2.
$$

Conversely, let $\xi=\sum_{j=1}^kx_j\otimes y_j$ be an element of an
extremal ray $R\subset C_1\otimes C_2$, where $x_j\in
C_1\setminus\{0\}$ and $y_j\in C_2\setminus\{0\}$ for all $j$. Then
$x_j\otimes y_j\in R$ for all $j$ because $x_j\otimes y_j\not=0$ for
all $j$ -- the tensor product of two nonzero vectors is a nonzero
vector. Consequently, there exist real
numbers $t_j>0$ with $x_j\otimes y_j=(t_j x_1)\otimes y_1$ for every
$j$. In particular, $R=\RR_+(x_1\otimes y_1)$. All we need to show is
that $x_1$ and $y_1$ are extremal generators of $C_1$ and $C_2$,
respectively. Without loss of generality we can assume that $x_1$ is
not an extremal generator of $C_1$. There is a segment $[u,v]\in C_1$
such that $x_1\in(u,v)$ and $0\notin\aff([u,v])$. But then the subset
$[u\otimes y_1,v\otimes y_1]\subset C_1\otimes C_2$ is a nondegenerate
segment and $0\notin\aff(u\otimes y_1,v\otimes y_1)$ -- one uses the
same fact on the tensor product of nonzero vectors. This contradicts the assumption that $x_1\otimes y_1\in R$.

\medskip\noindent(5) We have the mutually inverse linear maps:
\begin{align*}
&\alpha:\Hom(P,C)\to\Hom(C(P),C),\quad\alpha(h)((x,z))=h(z^{-1}x),\ (x,z)\in C(P)\setminus\{0\},\\
&\beta:\Hom(C(P),C)\to\Hom(P,C),\quad \beta(h)(x)=h((x,1)),\quad x\in P.
\end{align*}

\medskip\noindent(6) We have the following mutually inverse affine maps:
\begin{align*}
&\gamma:P^{\o}\to C(P)^{\o}_{[1]},\quad\gamma(h)((x,z))=z-z\cdot h(z^{-1}x),\ (x,z)\in C(P)\setminus\{0\},\\
&\delta:C(P)^{\o}_{[1]}\to P^{\o},\quad\delta(h)(x)=1-h((x,1)),\quad x\in P.
\end{align*}
That $\gamma$ and $\delta$ in fact evaluate in the right objects and have all necessary properties is straightorward.

For the second isomorphism, we note that $[0,1]=\RR_+\cap(1-\RR_+)$ and so (5), the previous isomorphism, and Proposition \ref{coneversion}(4) apply.
\end{proof}

\begin{remark}\label{cautionprojective}(1) Caution is needed in extending linear algebra facts to cones. For instance, the canonical map $C_1^{\o}\otimes C_2^{\o}\to(C_1\otimes C_2)^{\o}$ is in general \emph{not} an isomorphism: by Proposition \ref{conetensor}(4), the extremal rays of the first cone naturally correspond to the pairs of facets $F_1\subset C_1$ and $F_2\subset C_2$, while those of the second one are in bijective correspondence with the facets of $C_1\otimes C_2$. But there is no natural correspondence between the two sets. In fact, they can easily be non-bijective -- a combined effect of Proposition \ref{homproperties}(1) and the upcoming Proposition \ref{tensorproperties}(1).

(2) As a consequence of (6), for two isomorphic full dimensional
polytopes $P_1\cong P_2$ with $0\in\int(P_1)\cap\int(P_2)$, although
$P_1^{\o}$ and $P_2^{\o}$ are in general not isomorphic in $\Pol$, we
nonetheless have the isomorphism
$$
C(P_1^{\o})\cap((0,1)-C(P_1^{\o}))\cong C(P_2^{\o})\cap((0,1)-C(P_2^{\o})).
$$
In particular, we recover the standard fact that, when a full-dimensional polytope moves around the origin so that the origin stays in the interior of the polytope, the resulting dual polytopes are all projectively equivalent.
\end{remark}

For polytopes $P,Q,R$, a map $P\times Q\to R$ is called \emph{bi-affine} if, upon fixing one component, it is affine w.r.t. the other component.

\begin{proposition}\label{tensorproperties}
Let $P,Q,R$ be polytopes.
\begin{enumerate}
\item $C(P\otimes Q)\cong C(P)\otimes C(Q)$, or, equivalently,
$$
\vertex(P\otimes Q)=\{(v\otimes w,v,w)\ |\ v\in\vertex(P),\ w\in\vertex(Q)\}.
$$
\item
The bi-affine map $P\times Q\longrightarrow P\otimes Q$, $(x,y)\mapsto(x\otimes y,x,y)$, solves the following universal problem: any bi-affine map $P\times Q\longrightarrow R$ passes through a unique affine map $\phi$,
$$
\xymatrix{
P\times Q\ar@{=>}[r]\ar@{=>}[rd]_f&P\otimes Q\ar@{.>}[d]_{\circlearrowright\ \ }^{\exists!\phi}\\
&R\\
}\ .
$$
Equivalently, $\Hom(P\otimes Q, R)\cong \Hom(P,\Hom(Q,R))$, i.e., $\otimes, \Hom:\Pol\times\Pol\to\Pol$ form a pair of left and right adjoint functors.
\item
\begin{align*}
&\dim(P\otimes Q)=\dim P\dim Q+\dim P+\dim Q,\\
&P\otimes\{*\}\cong P,\ (P\otimes Q)\otimes R\cong P\otimes(Q\otimes R),\ P\otimes Q\cong Q\otimes P.\\
&P\otimes\join(Q,R)\cong\join(P\otimes Q,P\otimes R),\\
&\vertex(P\otimes Q)=\{(v\otimes w,v,w)\ |\ v\in\vertex(P),\ w\in\vertex(Q)\}.
\end{align*}
\end{enumerate}
\end{proposition}

\begin{proof} (1) We have
\begin{align*}
C(P)\otimes C(Q)=&\cone\{(x\otimes y)\ |\ x\in C(P),\ y\in C(Q)\}=\\
&\cone\{(x\otimes y)\ |\ x\in (P,1),\ y\in (Q,1)\}=\\
&\cone\{(u\otimes v,u,v,1)\ |\ u\in P,\ v\in Q\}=C(P\otimes Q).
\end{align*}
\medskip\noindent(2) A bi-affine map $\psi:\xymatrix{P\times Q\ar@{=>}[r]& R}$ gives rise to the following bilinear map:
\begin{align*}
&\xymatrix{ C(P)\times C(Q)\ar@{=>}[r]& C(R)},\\
&((u,x),(v,y))\mapsto\left(xy\cdot\psi\left(x^{-1}u,y^{-1}v\right),xy\right),\quad x,y>0.
\end{align*}
By Proposition \ref{conetensor}(1), we have the linear map:
\begin{align*}
&C(P)\otimes C(Q)\to C(R),\\
&(u,x)\otimes(v,y)\mapsto\left(xy\cdot\psi\left(x^{-1}u,y^{-1}v\right),xy\right),\quad x,y>0,
\end{align*}
which shows that the composite affine map
$$
P\otimes Q\cong(P\otimes Q,1)\hookrightarrow C(P\otimes Q)\cong C(P)\otimes C(Q)\to C(R),
$$
evaluates in $(R,1)$. The uniqueness is straightforward.

\medskip\noindent(3) This follows from (1) and the corresponding parts of Proposition \ref{conetensor}.
\end{proof}

For polytopes $P,Q,R$, the pairing
$$
\Hom(P,Q)\times\Hom(Q,R)\to\Hom(P,R),\qquad (f,g)\mapsto g\circ f,
$$
is clearly bi-affine. The pentagon and hexagon coherence conditions of the bifunctor $\otimes:\Pol\times\Pol\to\Pol$ are inherited from the corresponding conditions of the tensor product of vector spaces.  So Proposition \ref{tensorproperties} has the
following consequence.

\begin{corollary}\label{enriched}
$(\Pol,\otimes)$ is a symmetric closed monoidal category w.r.t. which $\Pol$
is self-enriched.
\end{corollary}

As a more computational application, we have the following.

\begin{corollary}\label{explicitcomputations} Let $P$ be a polytope
  and $n$ and $m$ be natural numbers. Then:
\begin{enumerate}
\item $\Hom(\Delta_n,P)\cong P^{n+1}$ and $\Delta_n\otimes P\cong P^{\join^{n+1}}$ -- the $(n+1)$-fold iteration of join, applied to $P$.
\item $\Hom(P,\Box_n)\cong\Diamond(P^{\o})^n$ for $P$ centrally symmetric w.r.t. $0$. In particular, $\Hom(\Box_m,\Box_n)\cong(\Diamond_{m+1})^n$.
\item $\Hom(\Box_m,\Diamond_n)\cong\Hom(\Box_{n-1},\Diamond_{m+1})$.
\end{enumerate}
\end{corollary}

\begin{proof} (1) Any map $\vertex(\Delta_n)\to P$ extends uniquely to an affine map $\Delta_n\to P$.

For the second isomorphism, one uses Proposition \ref{tensorproperties}(3), the third isomorphism, and the fact that $\Delta_n\cong(*)^{\join^{n+1}}$.

\medskip\noindent(2) This follows from Propositions \ref{homproperties}(3) and \ref{conetensor}(6), the 2nd isomorphism, and the fact that $P^{\o}$ is also symmetric w.r.t. the origin. (See Remark \ref{cautionprojective}(2).)

\medskip\noindent(3) One applies the conjunction $\otimes\dashv\Hom$ twice.
\end{proof}

We have not been able to find $\rank \ge2$ elements of $\vertex(\Hom(\Box_m,\Delta_n))$. More generally, one would wish to have Corollary \ref{explicitcomputations} completed by an explicit description of the hom-polytopes between the general regular polytopes. Dimensions 3 and 4 seem more challenging than $\dim\ge5$. Currently, even the 2-dimensional case, or just a satisfactory description of $\vertex(\Hom(P_n,P_m))$, is out of reach. Partial results in the latter direction are presented in Section \ref{regular}.

In the context of extremal maps from regular polytopes, it is
interesting to remark that the octahedron $\Diamond_3$ admits an affine embedding into \emph{any} simple 3-dimensional polytope $P$ so that the vertices of $\Diamond_3$ map to the boundary $\partial P$ \cite{Akopyan}.

\medskip We conclude the section with one application of the tensor product, yielding point configurations with interesting extremal properties.

\begin{example}\label{neighboring}
Let $P$ and $Q$ be polytopes in a vector space $E$.

\medskip\noindent(1) For a vertex $v\in P$ and an edge $[w_1,w_2]\subset Q$, the segment
$$
[(v\otimes w_1,v,w_1),(v\otimes w_2,v,w_2)]\subset E^{\otimes^2}\oplus E^{\oplus^2}
$$
is easily seen to be an edge of $P\otimes Q$.

\medskip\noindent(2) The tensor square $P^{\otimes^2}=P\otimes P$ admits less obvious edges:
$$
[(v\otimes v,v,v),(w\otimes w,w,w)]\subset P^{\otimes^2},\quad v,w\in\vertex(P),\ v\not=w.
$$

To show this, assume $v,w\in\vertex(P)$, $v\not=w$ and let $\phi_v,\phi_w:C(P)\to\RR_+$ be linear maps, vanishing on exactly $\cone(v')$ and $\cone(w')$, respectively, where $v'=(v,1)$ and $w'=(w,1)$. Then the linear map
$$
\xymatrix{
C(P^{\otimes^2})\ar[r]^{\cong}&C(P)^{\otimes^2}\ar[rrr]^{\phi_v\otimes\phi_w+\phi_w\otimes\phi_v}&&&\RR_+
}
$$
(the isomorphism on the left is from Proposition \ref{tensorproperties}(1)), vanishes
on exactly the 2-face
$$
\cone(v'\otimes v',w'\otimes w')\subset C(P^{\otimes^2}).
$$

\medskip As a consequence of the two types of edges, if $P$ has $m$ vertices and $n$ edges, then the tensor square $P^{\otimes^2}$ has at least $2mn+\frac{m(m-1)}2$ edges. This counting is, however, far from the complete list; an example is $(\Delta_n)^{\otimes^2}=\Delta_{n^2+2n}$.

\medskip\noindent(3) It follows from (2) that the following polytope is \emph{neighborly}:
$$
\conv\big((v\otimes v,v,v)\ |\ v\in\vertex(P)\big)\subset E^{\otimes^2}\oplus E^{\oplus^2},
$$
i.~e., any two vertices are joined by an edge. (In \cite[Ch.0]{ZiPOL} the property is called \emph{2-neighborly}.)

\medskip\noindent(4) A similar argument implies that for the unit Euclidean ball $B_d\subset\RR^d$ and any two points $v,w\in S^{d-1}=\partial B_d$, the segment
$$
[(v\otimes v,v,v),(w\otimes w,w,w)]\subset(B_d)^{\otimes^2}\quad (=\conv((x\otimes y,x,y)\ |\ x,y\in B_d))
$$
is an extremal subset (\cite[p.10]{KRIPO}).

\medskip\noindent(5) Similarly, the map
$$
\psi_{d-1}:S^{d-1}\to\RR^{d^2+2d},\quad v\mapsto(v\otimes v,v,v),
$$
has the following extremal property: for any system of points
$v_1,\ldots,v_n\in S^{d-1}$, their $\psi_{d-1}$-images are in convex
and neighborly position. This observation is, in a sense, weaker than (4) because $\conv(\Im\psi_{d-1})\subsetneqq (B_d)^{\otimes^2}$. (For instance, $(v\otimes w,v,w)\in(B_d)^{\otimes^2}\setminus\conv(\Im\psi_{d-1})$ for $v,w\in S^{d-1}$, $v\not=w$.) On the other hand, it yields the interesting embedding
$$
\psi:S^1\to\RR^4,\quad(\cos t,\sin t)\mapsto(\cos t,\sin t,\cos2t,\sin2t),
$$
which maps any number of points on the unit circle into a convex neighborly point configuration in $\RR^4$. In fact, $\psi$ is obtained from $\psi_1$ by the following series of affine transformations of $\RR^8$
\begin{align*}
&(\cos t,\sin t)\mapsto(\cos^2t,\sin^2t,\cos t\sin t,\cos t\sin t,\cos t,\sin t,\cos t,\sin t)\mapsto\\
&(\cos t,\sin t,\cos^2t,\sin^2t,\cos t\sin t)\mapsto(\cos t,\sin t,\cos^2t,1-\cos^2t,\sin2t)\mapsto\\
&(\cos t,\sin t,\cos^2t,\sin2t)\mapsto(\cos t,\sin t,\cos2t,\sin2t),
\end{align*}
implying an affine isomorphism $\conv(\Im\psi_1)\cong\conv(\Im\psi)$.
\end{example}

\section{Vertex factorizations}\label{vertices}

Let $A$ and $B$ be convex sets. A family of maps $\big(f_t\big)_{(-1,1)}\subset\Hom(A,B)$
is called an \emph{affine 1-family} if the map
$$
(-1,1)\to\Hom(A,B),\quad t\mapsto f_t,
$$
is injective and affine.

A map $f:P\to Q$ in $\Pol$ is called a \emph{vertex map}, or just a \emph{vertex}, if $f\in\vertex(\Hom(P,Q))$.

We will need the following obvious \emph{(affine) perturbation criteria} for vertices and interior points. Let $f:P\to Q$ be an affine map between two polytopes. Then:
\begin{enumerate}
\item[$(\pc_1)$]
$f$ is not a vertex if and only if there is an affine 1-family
$\big(f_t\big)_{(-1,1)}\subset\Hom(P,Q)$, with  $f_0=f$.
\item[$(\pc_2)$]
$f\in\int(\Hom(P,Q))$ if and only if for any affine 1-family $\big(f_t\big)_{(-1,1)}\subset\aff(P,Q)$ with $f_0=f$ there exists a real number $\varepsilon>0$ such that $\big(f_t\big)_{(-\varepsilon,\varepsilon)}\subset\Hom(P,Q)$
\end{enumerate}

Every map $f$ in $\Pol$ factors, uniquely up to the obvious equivalence, into a surjective and an injective map in $\Pol$: $f=f_{\inj}\circ f_{\surj}$.

A map $f:P\to Q$ is called a \emph{deflation} if it satisfies the conditions:
\begin{itemize}
\item[(i)] $f$ is surjective,
\item[(ii)] $f$ is a vertex,
\item[(iii)] for any vertex $v\in P$ either $f(v)\in\vertex(Q)$ or $f(v)\in\int(Q)$.
\end{itemize}

Simple examples of deflations are provided by the parallel projections
$\Box_n\to\Box_m$ along the subspace $\RR_+e_{m+1}+\cdots+\RR_+e_n\subset\RR^n$, where we assume $m<n$.

The \emph{rank} of a map $f$ in $\Pol$ is defined by $\rank f=\dim\Im(f)$.

\begin{theorem}\label{vertexcomposite}
Let $f:P\to Q$ be a map in $\Pol$.
\begin{enumerate}
\item
If $f$ is a vertex then so are $f_{\inj}$ and $f_{\surj}$.
\item If $\vertex(f(P))\subset\vertex(Q)$ and $f_{\surj}$ is a deflation then $f$ is also a vertex map.
\item There are examples of non-vertex maps $f$ with $f_{\inj}$ a vertex map and $f_{\surj}$ a deflation. In particular, the converse to the implication in (1) is not true.
\end{enumerate}
\end{theorem}

\begin{proof} We can assume $f_{\surj}:P\to f(P)$.

(1) (Using $(\pc_1)$.) If $f_{\inj}:f(P)\to Q$ is not a vertex then there exists an affine 1-family $(\psi_t)_{(-1,1)}\subset\Hom(f(P),Q)$ with $\psi_0=f_{\inj}$. Then the resulting affine 1-family $\big(\psi_t\circ f_{\surj}\big)_{(-1,1)}\subset\Hom(P,Q)$ has $\psi_0\circ f_{\surj}=f$, a contradiction.

If $f_{\surj}$ is not vertex then there is an affine 1-family $(\phi_t)_{(-1,1)}\subset\Hom(P,f(P))$ with $\phi_0=f_{\surj}$. So we get the affine 1-family $\big(f_{\inj}\circ\phi_t\big)_{(-1,1)}\subset\Hom(P,Q)$ with $f_{\inj}\circ\phi_0=f$, again a contradiction.

\medskip\noindent(2) Observe that, in view of $(\pc_1)$, the
condition $\vertex(f(P))\subset\vertex(Q)$ already implies that $f_{\inj}$ is a vertex map. But we do not use this explicitly.

Assume to the contrary that $f$ is not a vertex. By $(\pc_1)$, there is an affine 1-family $(f_t)_{(-1,1)}\subset\Hom(P,Q)$ with $f_0=f$.

We will show that the maps
$$
\phi_v:(-1,1)\to Q,\qquad t\mapsto f_t(v),
$$
are constant for \emph{all} $v\in\vertex(P)$. Since this implies that the map
$$
(-1,1)\to\Hom(P,Q),\qquad t\mapsto f_t,
$$
is constant, we have the desired contradiction.

The maps
$$
\phi_v:(-1,1)\to Q,\qquad t\mapsto f_t(v),
$$
are constant for the vertices $v\in\vertex(P)$ with $f(v)\in\vertex(Q)$.

Assume $v_1,\ldots,v_k$ are the vertices of $P$ for which the maps
$$
\phi_{v_1},\ldots,\phi_{v_k}:(-1,1)\to Q,\qquad t\mapsto f_t(v_k),
$$
are not constant. By the hypothesis, $f(v_1),\ldots,f(v_k)\in\int(f(P))$.

Fix a $(\dim Q-\rank f)$-dimensional affine subspace $H\subset\aff(Q)$ that is parallel to none of the intervals
$$
\bigcup_{t\in(-1,1)}f_t(v_i)\subset Q,\quad i=1,\ldots,k,
$$
and consider the parallel projection $\pi:Q\to\aff(\Im(f))$ along $H$. We get an affine family $(\pi\circ f_t)_{(-1,1)}\subset=\aff(P,\Im(f))$ with $\pi\circ f_0=f_{\surj}$. By $(\pc_2)$, there exists a real number $\varepsilon>0$ which defines the affine family $(\pi\circ f_t)_{(-\varepsilon,\varepsilon)}\subset\Hom(P,\Im(f))$, contradicting the assumption that $f_{\surj}$ is vertex.

\medskip\noindent(3) Examples of non-vertex maps $f$ for which $f_{\inj}$ is a vertex and
$f_{\surj}$ is a deflation will be presented in Example \ref{nonvertexfactorization}.
\end{proof}

Let $P$ be a polytope and $0\le r<\dim P$. Consider a family $\Gamma=\{G_1,\ldots,G_k\}$, satisfying the conditions:
\begin{itemize}
\item[(i)]
$G_i\subset P$ is a face and $\dim G_i>0$, $i=1,\ldots,k$,
\item[(ii)]
$\codim(\lin(G_1)+\cdots+\lin(G_k))=r$,
\item[(iii)]
$(G_i+\lin(G_1)+\cdots+\lin(G_k))\cap P=G_i$, $i=1,\ldots,k$,
\item[(iv)] the family of faces $\{G_1,\ldots,G_k\}$ is maximal w.r.t inclusion among the families satisfying the first three conditions.
\end{itemize}
The linear hulls above are taken in the ambient vector space and the codimension is understood relative to $\lin(P)$.

To $\Gamma$ we associate a surjective affine map of rank $r$ from $P$ as follows. We can assume $0\in P$. Consider a linear map $f:\lin(P)\to E$ with $\dim E=r$ and $\ker(f)=\lin(G_1)+\cdots+\lin(G_k)$. For the restriction $f|_P$ we will use the notation $f_\Gamma$.

A map $P\to Q$ in $\Pol$ is called a \emph{face-collapse of rank $r$}
if it is of the type $f_{\Gamma}$ up to isomorphism in the comma category $P\downarrow\Pol$.

\medskip The following example shows that neither of the conditions (iii) and (iv) above can be dropped in the definition of face-collapses. Let $P_6$ be a regular hexagon with vertices $v_1,\ldots,v_6$. Then the pair of facets $\Gamma=\{[v_1,v_2]\times[0,1],\ [v_4,v_5]\times[0,1]\}$ of the unit prism $P_6\times[0,1]$ satisfies (i-iv). On the one hand, if one drops (iii), then $\Gamma$ ceases to be maximal. On the other hand, if one drops (iv), then $\Gamma'=\{[v_1,v_2]\times[0,1]\}$ satisfies (i,ii,iii), but one cannot collapse the facet $[v_1,v_2]\times[0,1]$ into a point by an affine map without also collapsing the facet $[v_4,v_5]\times[0,1]$:

\medskip

\begin{figure}
\tikzstyle{vertex}=[circle,draw,fill,inner sep=2pt]
\begin{tikzpicture}[scale=0.5]
\node at (0,0)[vertex,label=below:$v_6$]{};
\node at (0.8,3.3)[vertex,label=right:$v_5$]{};
\node at (4,0.5)[vertex,label=below:$v_1$]{};
\node at (7.2,2.9)[vertex,label=below:$v_2$]{};
\node at (0,7)[vertex]{};
\node at (-.5,3.6)[vertex]{};
\node at (3.6,4.4)[vertex]{};
\node at (4.7,6.1)[vertex,label=above right:$v_4$]{};
\node at (4.2,10.4)[vertex]{};
\node at (6.9,7.2)[vertex]{};
\node at (7.6,10.7)[vertex]{};
\node at (7.9,6.4)[vertex,label=right:$v_3$]{};
\draw[fill=lightgray] (0,7) -- (4.2,10.4) -- (4.7,6.1) -- (0.8,3.3) -- (0,7);
\draw[fill=lightgray] (4,0.5) -- (3.6,4.4) -- (6.9,7.2) -- (7.2,2.9) -- (4,0.5);
\draw (0,0) -- (4,0.5);
\draw (0,0) -- (-0.5, 3.6);
\draw (-0.5,3.6) -- (0,7);
\draw (4.2, 10.4) -- (7.6, 10.7);
\draw (7.2, 2.9) -- (7.9, 6.4);
\draw (7.9, 6.4) -- (7.6, 10.7);
\draw (7.6, 10.7) -- (6.9, 7.2);
\draw (3.6, 4.4) --  (-0.5, 3.6);
\draw[dashed] (0,0) -- (0.8, 3.3);
\draw[dashed] (4.7, 6.1) -- (7.9, 6.4);
\end{tikzpicture}
\caption{A polytope illustrating the definition of face-collapse}
\end{figure}
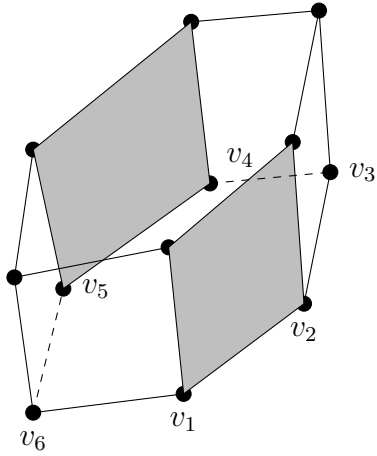

\begin{theorem}\label{fGamma} We have the following proper containments:
\begin{itemize}
\item[]Deflations $\subsetneqq$ Face-Collapses $\subsetneqq$ Surjective Vertex Maps.
\end{itemize}
\end{theorem}

\begin{proof}
First we show that face-collapses are vertex maps.

Let $P\in\Pol$ and $\Gamma=\{G_1,\ldots,G_k\}$ be a system of faces of $P$ satisfying the condition (i--iv) above.
Without loss of generality we can assume $0\in P$. Denote $H=\lin(G_1)+\cdots+\lin(G_k)$.

First we observe that $f(G_i)\in\vertex(f(P))$ for every
$i=1,\ldots,k$. In fact, $f(G_i)\subset\partial f(P)$ by (iii) and
$\dim f(G_i)=0$. But then, if $f(G_i)\in\int(\gamma)$ for some
positive dimensional face $\gamma\subset f(P)$, we have
$\int(G_i)\subset\int\big(f_\Gamma^{-1}(\gamma)\big)$ and
$f_\Gamma^{-1}(\gamma)$ is a face of $P$ with $G_i\subsetneqq f_\Gamma^{-1}(\gamma)$. This is impossible because $G_i\subset P$ is a face.

Suppose $f_\Gamma$ is not a vertex. By $(\pc_1)$, there exists an affine 1-family $(f_t)_{(-1,1)}\subset\Hom(P,f(P))$ with $f_0=f_\Gamma$. The condition $f(G_i)\in\vertex(f(P))$ forces $f_t|_{G_i}=f|_{G_i}$ for all $t\in(-1,1)$ and $i=1,\ldots,k$. In other words, the 1-family $(f_t)_{(-1,1)}$ is constant on every $G_i$. Let $\{w_{k+1},\ldots,w_{k+l}\}=\vertex(f(P))\setminus\{f(G_1),\ldots,f(G_k)\}$. Fix $v_{k+1},\ldots,v_{k+l}\in\vertex(P)$ with $f(v_j)=w_j$, $j=k+1,\ldots,k+l$. (The $v_j$ are uniquely determined, but we do not need this.) Using again $(\pc_1)$, the 1-family $(f_t)_{(-1,1)}$ is constant on the $v_j$. So $(f_t)_{(-1,1)}$ is constant on $\conv(G_1,\ldots,G_k,v_{k+1},\ldots,v_{k+l})$. But, by (ii), the latter is a full dimensional subpolytope of $P$ and that forces our 1-family to be constant on $P$ -- a contradiction.

\medskip Next we show that deflations are face-collapses.

Consider a rank $r$ deflation $f\in\vertex\big(\Hom(P,f(P)\big)$. We
can assume that $0\in P$ and that $f(0)=0$; i.e., $f$ is the restriction of a (unique) linear map $h:\lin(P)\to\lin(f(P))$. Let $w_1,\ldots,w_k\in f(P)$ be the vertices with $\dim f^{-1}(w_i)>0$ for all $i=1,\ldots,k$. It is enough to show the following

\medskip\noindent\emph{Claim.} The faces $G_i=f^{-1}(w_i)\subset P$, $i=1,\ldots,k,$ form a system satisfying the conditions (i--iv) with $\ker(h)=\lin(G_1)+\cdots+\lin(G_k)$.

\medskip The conditions (i,iii,iv) are straightforward. Since $\lin(G_1)+\cdots+\lin(G_k)\subset\ker(h)$, we only need to verify (ii).

Assume to the contrary $\codim\big(\lin(G_1)+\cdots+\lin(G_k)\big)>r$. Let $\{w_{k+1},\ldots,w_{k+l}\}=\vertex(f(P))\setminus\{w_1,\ldots,w_k\}$. There are (uniquely determined) vertices $v_{k+1},\ldots,v_{k+l}\in P$ with $f(v_j)=w_j$ for $j=k+1,\ldots,k+l$. Our assumption implies
$$
\dim\conv(G_1,\ldots,G_k,v_{k+1},\ldots,v_{k+l})<\dim P.
$$
Without loss of generality we can additionally assume
$$
0\in\conv(G_1,\ldots,G_k,v_{k+1},\ldots,v_{k+l}).
$$
Pick a basis $B\subset\lin(P)$, restricting to a basis
$B_0\subset\lin(G_1,\ldots,G_k,v_{k+1},\ldots,v_{k+l})$, and a basis
$B'\subset\lin(f(P))$. Let $M$ be the matrix of $h$ with respect to the bases $B$ and $B'$ so that the first $\#(B_0)$ rows correspond to the elements of $B_0$. Consider the affine 1-family
$(h_t)_{(-1,1)}\subset\Hom(\lin(P),\lin(f(P)))$, where the matrix of $h_t$ in the bases $B$ and $B'$ is obtained from $M$ by adding $(t,\ldots,t)\in\RR^{\#B'}$ to each of the last $\#(B\setminus B_0)$ rows. Then, because $f$ is a deflation, there exists $0<\varepsilon<1$ such that $h_t(P)\subset f(P)$ for all $-\varepsilon<t<\varepsilon$ and this produces an affine 1-family in $\Hom(P,f(P))$, containing $f$ -- a contradiction by $(\pc_1)$.

\medskip Examples \ref{propercollapses} and \ref{propersurjvertices} in the next section will show that both inclusions in Theorem \ref{fGamma} are proper inclusions.
\end{proof}

\begin{corollary}\label{rank1}
\begin{enumerate}
\item A rank $1$ surjective map is a vertex if and only if it is a face-collapse if and only if it is a deflation.
\item Every system $\Gamma$ of faces of $P$ satisfying \emph{(i-iv)} and $\rank(f_\Gamma)=1$ has at most two elements. If $P$ is centrally symmetric then such systems $\Gamma$ are exactly the pairs of opposite facets of $P$.
\item A rank one map $f$ is a vertex iff $f_{\surj}$ and $f_{\inj}$ are vertices.
\item For a polygon $P$ with $l$ edges, of which $m$ pairs are parallel, and a polytope $Q$ with $n$ vertices, the number of rank $1$ vertex maps $P\to Q$ is
$(l-m)n(n-1)$.
\end{enumerate}
\end{corollary}

\begin{proof}
(1) holds true because of the first inclusion in Theorem \ref{fGamma} and the fact that the boundary of a segment is just the vertices of the segment .

\medskip\noindent(2) If $\{u,v\}=\vertex f_\Gamma(P)$ and $\dim
f^{-1}(u),\dim^{-1}(v)>0$ then $\Gamma=\{f^{-1}(u),f^{-1}(v)\}$;
otherwise $\Gamma$ consists of a single facet of $P$. The conclusion for centrally symmetric polytopes is obvious.

\medskip\noindent(3) From $(\pc_1)$, the injective vertex maps from a
segment to $Q$ are the maps that send the two vertices of the segment to
distinct vertices of $Q$. So the claim follows from (1) and Theorem
\ref{vertexcomposite}(1,2).

\medskip\noindent (4) In view of (1) and (2), $(l-m)$ is the number of
surjective vertex maps from $P$ to a segment. The count follows from (3) and the observation above on injective vertex maps from a segment.
\end{proof}

\section{Examples of vertex maps}\label{examples}

In this section we present various constructions of vertex maps exhibiting interesting phenomena, some of them mentioned in the previous sections.

\begin{example}[\emph{Vertex factorization of non-vertex maps}]\label{nonvertexfactorization}
For a map $f:P\to Q$ in $\Pol$, if $1+\dim P$ affinely independent
vertices of $P$ map to vertices of $Q$ then $f$ is a vertex. In fact,
if there were an affine 1-family $(f_t)_{(-1,1)}\subset\Hom(P,Q)$ with
$f_0=f$, then the family would be constant on the distinguished
vertices of $P$. The affine hull of these vertices is all of
$\aff(P)$. Consequently, $(f_t)_{(-1,1)}$ must be constant on $P$ -- a
contradiction. This simple observation, together with Theorem
\ref{vertexcomposite}(2), suggests that to look for non-vertex maps
$f:P\to Q$ with $f_{\surj}$ and $f_{\inj}$ both vertices, we should
consider maps that keep the vertices of $\Im f$ away from
$\vertex(Q)$. We first give a general construction and then construct examples of vertex factorizations of non-vertex maps.

Let $Q\subset\RR^3$ be a 3-polytope and $H\subset\RR^3$ be an affine plane such that: (i) the intersection $P=Q\cap H$ is a polygon with $2n\ge6$ vertices, and (ii) there exist $w_1,\ldots,w_{2n}\in\vertex(Q)$ such that every open interval $(w_i,w_{i+1})$ contains exactly one vertex $v_i\in\vertex(P)$, the indexing being mod $2n$. Observe that the segments $[w_i,w_{i+1}]$ are necessarily edges of $Q$.

\medskip\noindent\emph{Claim}. The identity embedding $\iota:P\to Q$ is a vertex of $\Hom(P,Q)$.

\medskip
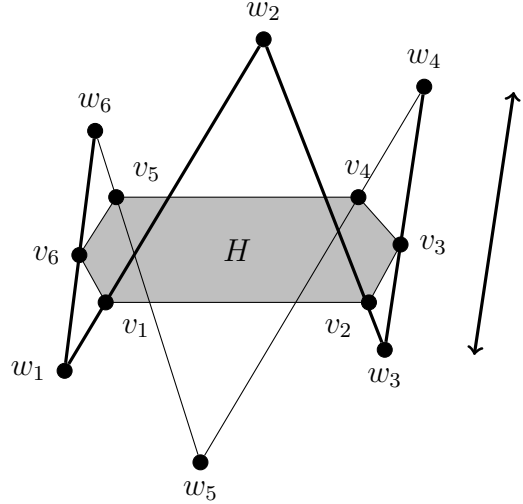
\begin{figure}
\tikzstyle{vertex}=[circle,draw,fill,inner sep=2pt]
\begin{tikzpicture}[scale=0.7]
\draw[fill=lightgray] (0,0) -- (5,0) -- (5.6,1.1) -- (4.8,2) -- (0.2,2) -- (-0.5,0.9) -- (0,0);
\node at (0,0)[vertex,label=below right:$v_1$](v1){};
\node at (5,0)[vertex,label=below left:$v_2$](v2){};
\node at (5.6,1.1)[vertex,label=right:$v_3$](v3){};
\node at (4.8,2)[vertex,label=above:$v_4$](v4){};
\node at (0.2,2)[vertex,label=above right:$v_5$](v5){};
\node at (-0.5,0.9)[vertex,label=left:$v_6$](v6){};
\node at (-0.78,-1.3)[vertex,label=left:$w_1$](w1){};
\node at (3,5)[vertex,label=above:$w_2$](w2){};
\node at (5.3,-0.9)[vertex,label=below:$w_3$](w3){};
\node at (6.05,4.1)[vertex,label=above:$w_4$](w4){};
\node at (1.8,-3.04)[vertex,label=below:$w_5$](w5){};
\node at (-0.2,3.26)[vertex,label=above:$w_6$](w6){};
\draw[very thick] (w6) -- (w1) -- (w2) -- (w3) -- (w4);
\draw (w4) -- (w5) -- (w6);
\draw[<->,very thick](7,-1) -- (7.75,4);
\draw node at (2.5,1){$H$};
\end{tikzpicture}
\caption{The construction of Example \ref{nonvertexfactorization} with
$n=3$}
\end{figure}

\medskip\noindent The two groups of vertices $w_1,w_3,\ldots,w_{2n-1}$
and $w_2,w_4,\ldots,w_{2n}$ are separated by $H$. To see that $\iota$
is a vertex,  we apply the following \emph{sliding argument}. Assume
$\iota\notin\vertex(\Hom(P,Q))$. Pick a vertex of
$\sigma\in\Hom(P,Q)$, which belongs to the same minimal face of
$\Hom(P,Q)$ as $\iota$. By sliding $\iota$ along the segment
$[\iota,\sigma]$ one gets a 1-parameter family
$(\iota_t)_{[0,1]}\subset\Hom(P,Q)$ with $\iota_0=\iota$ and
$\iota_1=\sigma$ such that, for every $t\in[0,1]$, the image of
$\iota_t$ is a polygon, isomorphic to $P$ and with vertices
$v_{ti}\in[w_i,w_{i+1}]$. In fact, the `vertex of $\iota_t(P)$
$\leftrightarrow$ edge of $Q$' incidence table remains constant in the
process of sliding for $t\in[0,1)$. So $\dim(\Im\iota_t)=2$ for
$t\in[0,1)$. But $\dim\Im\sigma=2$ as well because no affine line can
intersect all segments $[w_i,w_{i+1}]$ simultaneously -- this is where
we use the inequality $n\ge6$.\footnote{For $n=2$ the construction
  does not go through; an example is a tetrahedron and a plane,
  intersecting the tetrahedron in a parallelogram -- the identity
  embedding of the parallelogram into the tetrahedron is not a
  vertex.} In particular, $\Im(\iota_t)\cong P$ for all
$t\in[0,1]$. Since none of the mentioned incidences is lost for $t=1$
and $\sigma$ is a vertex, $\sigma$ belongs to more facets of
$\Hom(P,Q)$ than $\iota$. So the condition $v_{1i}\in[w_i,w_{i+1}]$
for $i=1,\ldots,2n$ forces $v_{1j}=w_j$ or $v_{1j}=w_{j+1}$ for some
$j$. But then one easily deduces $w_1,\ldots,w_{2n}\in\Im\sigma$,
forcing $\dim(\Im\sigma)=3$. This contradicts the fact that
$\Im\sigma\cong P$, proving the claim.

Consider the special case of the construction above, using the regular $n$-gon $P_n$ with $n$ even:
\begin{align*}
Q=\conv(&\hat P_n,\check P_n)\subset\CC\oplus\RR,\quad \text{with}\\
&\hat P_n=(P_n,-1)\subset\CC\oplus\RR,\\
&\check P_n=(\eta P_n,1)\subset\CC\oplus\RR,\\
&\eta=\cos(\pi/n)+\sin(\pi/n)i.
\end{align*}

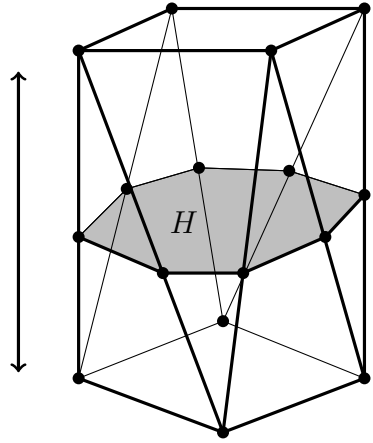
\begin{figure}
\begin{tikzpicture}[scale=0.4]
\draw[fill=lightgray] (0,6.5) -- (1.6,8.1) -- (4.0,8.8) -- (7.0,8.7)
-- (9.5,7.9) -- (8.2,6.5) -- (5.47,5.3) -- (2.8,5.3) -- (0,6.5);
\fill(0,1.8) circle(0.2);
\fill(4.8,0) circle(0.2);
\fill(9.5,1.8) circle(0.2);
\fill(4.8,3.7) circle(0.2);
\fill(0,6.5) circle(0.2);
\fill(2.8,5.3) circle(0.2);
\fill(5.47,5.3) circle(0.2);
\fill(8.2,6.5) circle(0.2);
\fill(9.5,7.9) circle(0.2);
\fill(7.0,8.7) circle(0.2);
\fill(4.0,8.8) circle(0.2);
\fill(1.6,8.1) circle(0.2);
\fill(0,12.7) circle(0.2);
\fill(6.4,12.7) circle(0.2);
\fill(9.5,14.1) circle(0.2);
\fill(3.1,14.1) circle(0.2);
\draw[very thick] (0,1.8) -- (4.8,0);
\draw[very thick] (4.8,0) -- (9.5,1.8);
\draw[very thick] (9.5,1.8) -- (9.5,14.1);
\draw[very thick] (9.5,14.1) -- (3.1,14.1);
\draw[very thick] (3.1,14.1) -- (0,12.7);
\draw[very thick] (0,12.7) -- (0,1.8);
\draw[very thick] (0,6.5) -- (2.8,5.3);
\draw[very thick] (2.8,5.3) -- (5.47,5.3);
\draw[very thick] (5.47,5.3) -- (8.2,6.5);
\draw[very thick] (8.2,6.5) -- (9.5,7.9);
\draw[very thick] (0,12.7) -- (6.4,12.7);
\draw[very thick] (6.4,12.7) -- (9.5,14.1);
\draw[very thick] (0,12.7) -- (4.8,0);
\draw[very thick] (4.8,0) -- (6.4,12.7);
\draw[very thick] (6.4,12.7) -- (9.5,1.8);
\draw (0,1.8) -- (4.8,3.7);
\draw (4.8,3.7) -- (9.5,1.8);
\draw (0,6.5) -- (1.6,8.1);
\draw (1.6,8.1) -- (4.0,8.8);
\draw (4.0,8.8) -- (7.0,8.7);
\draw (7.0,8.7) -- (9.5,7.9);
\draw (0,1.8) -- (3.1,14.1);
\draw (3.1,14.1) -- (4.8,3.7);
\draw (4.8,3.7) -- (9.5,14.1);
\draw[<->,very thick](-2,2) -- (-2,12);
\draw node at (3.5,7){$H$};
\end{tikzpicture}
\caption{The construction of Example \ref{nonvertexfactorization} with
$P_n$ regular, $n=4$}
\end{figure}

\medskip\noindent For every $h\in(-1,1)$, the polygon
$P_h=Q\cap(\CC,h)$, is a \emph{centrally symmetric} $2n$-gon and the
identity embedding $\iota_h:P_h\to Q$ is a vertex of $\Hom(P_h,Q)$. We
can choose surjective maps $\rho_h:\Diamond_n\to P_h$, $h\in(-1,1)$ so
that (i) for every $h$, vertices map to vertices and (ii) the maps
$\rho_h$, viewed as elements of $\aff(\RR^n,\CC\oplus\RR)$,
continuously depend on $h$. Because each $\rho_h$ maps vertices to vertices, we have $\rho_h\in\vertex(\Diamond_n,P_h)$ for every $h$. Moreover, $\rho_h$ is a deflation for every $h$.

It is important that the assignment $h\mapsto\rho_h$ is not just continuous, but even an affine map $(-1,1)\to\aff(\RR^n,\CC\oplus\RR)$.

Summarizing, for every $h\in(-1,1)$, the composite map $f_h=\iota_h\circ\rho_h$ has the surjective factor $(f_h)_{\surj}=\rho_h$ a deflation and the injective factor $(f_h)_{\inj}=\iota_h$ a vertex map. Yet,
$$
f_h\notin\vertex(\Hom(\Diamond_n,Q)).
$$
For simplicity of notation we can restrict to the case $h=0$, and then $f_0$ fits into the affine 1-family $(f_h)_{h\in(-1,1)}\subset\Hom(\Diamond_n,Q)$.

\medskip From the construction above one can derive new examples of vertex factorizations of non-vertex maps as follows. Observe that, for every $h\in(-1,1)$, there exists a deflation $\tau_h:\Box_n\to P_h$. So the composite maps $g_h=\iota_h\circ\tau_h$ do not belong to $\vertex(\Hom(\Box_n,Q))$, yet $(g_h)_{\surj}$ is a deflation and $(g_h)_{\inj}$ is a vertex map.
\end{example}

\begin{example}[\emph{Face-collapses that are not deflations}]\label{propercollapses}
Not all face-collapses are deflations. Here is an example. Let $\{e_1,e_2,e_3\}\subset\RR^3$ be the standard basis and consider the parallel projection in the direction of $-e_3$:
$$
\pi:\conv(0,2e_1,2e_2,e_3,e_1+e_3,e_2+e_3)\to\conv(0,2e_1,2e_2)
$$

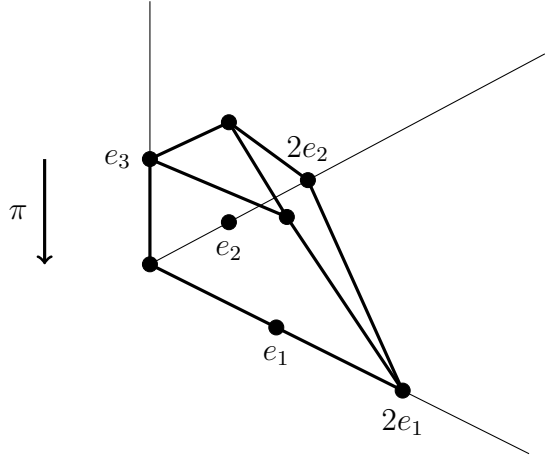
\begin{figure}
\tikzstyle{vertex}=[circle,draw,fill,inner sep=2pt]
\begin{tikzpicture}[scale=0.7]
\node at (0,0)[vertex]{};
\node at (2.4,-1.2) [vertex, label=below:$e_1$]{};
\node at (4.8,-2.4) [vertex, label=below:$2e_1$]{};
\node at (1.5,0.8) [vertex, label=below:$e_2$]{};
\node at (3,1.6) [vertex, label=above:$2e_2$]{};
\node at (0,2) [vertex, label=left:$e_3$]{};
\node at (2.6,0.9) [vertex]{};
\node at (1.5,2.7) [vertex]{};

\draw (0,0) -- (0,5);
\draw (0,0) -- (7.5,4);
\draw (0,0) -- (7.2,-3.6);
\draw[very thick] (0,0) -- (4.8, -2.4);
\draw[very thick] (0,0) -- (0,2);
\draw[very thick] (0,2) -- (2.6,0.9);
\draw[very thick] (2.6,0.9) -- (4.8,-2.4);
\draw[very thick] (0,2) -- (1.5,2.7);
\draw[very thick] (1.5,2.7) -- (3,1.6);
\draw[very thick] (3,1.6) -- (4.8,-2.4);
\draw[very thick] (1.5,2.7) -- (2.6,0.9);
\draw[->, very thick] (-2,2) -- (-2,0);
\draw (-2.5,1) node{$\pi$};
\end{tikzpicture}
\caption{A face-collapse that is not a deflation}
\end{figure}

\medskip\noindent We have $\pi=f_{\Gamma}$ for $\Gamma=\{[0,e_3]\}$. However, $\pi$ is not a deflation because $\pi(e_1+e_3)$ and $\pi(e_2+e_3)$ are neither vertices not interior points of the target polytope.
\end{example}

\begin{example}[\emph{Not all surjective vertex maps are face-collapses}]\label{propersurjvertices}
By a slight modification of the map above, we get an example of a surjective vertex map which is not a face-collapse. Consider the parallel projection in the direction $-e_3$:
$$
\rho:\conv(0,2e_1,2e_2,e_1+e_3,e_2+e_3)\to\conv(0,2e_1,2e_2)
$$

\medskip
\begin{figure}
\tikzstyle{vertex}=[circle,draw,fill,inner sep=2pt]
\begin{tikzpicture}[scale=0.7]
\node at (0,0)[vertex]{};
\node at (2.4,-0.6) [vertex, label=below:$e_1$]{};
\node at (4.8,-1.2) [vertex, label=below:$2e_1$]{};
\node at (1.5,0.8) [vertex, label=above:$e_2$]{};
\node at (3,1.6) [vertex, label=above:$2e_2$]{};
\node at (0,2) [vertex, label=left:$e_3$]{};
\node at (2.6,0.9) [vertex]{};
\node at (1.5,2.7) [vertex]{};

\draw (0,0) -- (0,5);
\draw (0,0) -- (7.5,4);
\draw (0,0) -- (7.2,-1.8);
\draw[very thick] (0,0) -- (4.8, -1.2);
\draw[very thick] (0,0) -- (1.5,2.7);
\draw[very thick] (0,0) -- (2.6,0.9);
\draw[very thick] (2.6,0.9) -- (4.8,-1.2);
\draw[very thick] (1.5,2.7) -- (3,1.6);
\draw[very thick] (3,1.6) -- (4.8,-1.2);
\draw[very thick] (1.5,2.7) -- (2.6,0.9);
\draw[->, very thick] (-2,2) -- (-2,0);
\draw (-2.5,1) node{$\rho$};
\end{tikzpicture}
\caption{A surjective vertex map that is not a face-collapse}
\end{figure}
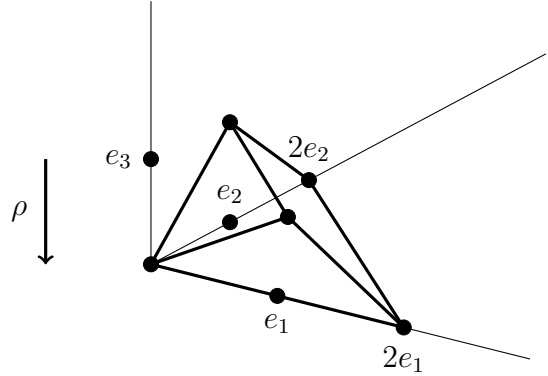

\medskip\noindent This projection cannot be a face-collapse because
none of the positive dimensional faces of the source polytope is
collapsed into a point. To show that $\rho$ is a vertex, assume there
is an affine 1-family $(\rho_t)_{(-1,1)}$ with $\rho_0=\rho$. Then it
must be constant on the vertices of $\conv(0,2e_1,2e_2)$. But it must
also be constant on $e_1+e_3$ and $e_2+e_3$. In fact, if the family is not constant there then $\rho_t(e_1+e_3)$ and $\rho_t(e_2+e_3)$ must trace out parallel intervals as $t$ varies over $(-1,1)$ (one uses barycentric coordinates). But, on the other hand, these trajectories must be confined to the non-parallel edges $[0,2e_1]$ and $[0,2e_2]$ of the target polytope. So the family $(\rho_t)_{(-1,1)}$ is constant on all vertices of the source polytope, forcing the family to be constant -- a contradiction.
\end{example}

\begin{example}[\emph{Many incident facets}]\label{onmanyfacets} As we know, $\dim(\Hom(P_n,P_n))=6$ for
  the regular $n$-gon, $n\in\NN$. By
  Proposition \ref{homproperties}(1), any automorphism $P_n\to P_n$
  sits on $2n$ facets of $\Hom(P_n,P_n)$. These $2n$ facets can be
  split into two groups of $n$ facets, each defining an edge of
  $\Hom(P_n,P_n)$. For the identity map, these groups are specified as
  follows.

\begin{align*}
&\FF_1=\left\{H\big(\iota(\zeta_n^k)),[\zeta_n^k,\zeta_n^{k+1}]\big)\subset\Hom(P_n,P_n)\ |\ k\in\ZZ\right\},\\
&\FF_2=\left\{H\big(\iota(\zeta_n^k)),[\zeta_n^{k-1},\zeta_n^k]\big)\subset\Hom(P_n,P_n)\ |\ k\in\ZZ\right\},
\end{align*}
notation as in Proposition \ref{homproperties}(1). The edges are, correspondingly,
\begin{align*}
&E_1=\bigcap_{\FF_1}F=\left\{\iota_t:P_n\to P_n,\quad \zeta_n^k\mapsto\left(\frac12-\frac t2\right)\zeta_n^k+\left(\frac12+\frac t2\right)\zeta_n^{k+1}\right\}_{t\in[-1,1]},\ \text{and}\\
&E_2=\bigcap_{\FF_2}F=\left\{\rho_t:P_n\to P_n,\quad \zeta_n^k\mapsto\left(\frac12-\frac t2\right)\zeta_n^{k-1}+\left(\frac12+\frac t2\right)\zeta_n^k\right\}_{t\in[-1,1]}.
\end{align*}
We leave to the reader to check that $E_1$ and $E_2$ are in fact the edges, joining the identity map with the rotations by $2\pi/n$ and $-2\pi/n$, respectively.

\begin{figure}
\begin{tikzpicture}[scale=2];
\fill (0,1) circle(0.035);
\fill (-.951,.309) circle(0.035);
\fill (-.588,-.809) circle(0.035);
\fill (.588,-.809) circle(0.035);
\fill (.951,.309) circle(0.035);
\fill (-.238,.827) circle(0.035);
\fill (-.860,-.030) circle(0.035);
\fill (-.294,-.809) circle(0.035);
\fill (.679,-.530) circle(0.035);
\fill (.713,.482) circle(0.035);
\draw[very thick] (0,1) -- (-.951,.309) -- (-.951,.309) --
(-.588,-.809) -- (.588,-.809) -- (.951,.309) -- (0,1);
\draw[fill=lightgray] (-.238,.827) -- (-.860,-.030) -- (-.294,-.809) -- (.679,-.530)
-- (.713,.482) -- (-.238,.827);
\draw[->] (-.9,.45) -- (-.142,1);
\draw[->] (.1,1.05) -- (.8,.54);
\draw[->] (1.05,.269) -- (.7,-.809);
\draw[->] (.5,-.95) -- (-.5,-.95);
\draw[->] (-.7,-.809) -- (-1.05,.269);
\end{tikzpicture}
\caption{Two edges of $\Hom(P_n,P_n)$ incident to the identity map, $n=6$}
\end{figure}
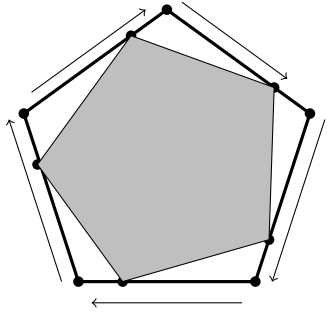
\end{example}

\begin{example}[\emph{Gaps in ranks}]\label{bisimplexinsimplex} For two polytopes $P$ and $Q$, the range of possible ranks $r$ of the vertices of $\Hom(P,Q)$ is
$$
0\le r\le\min(\dim P,\dim Q).
$$
For any polytopes $P$ and $Q$, the polytope $\Hom(P,Q)$ always has rank 0 vertices. If $\dim P,\dim Q>0$ then there are also rank 1 vertices, explicitly described in Corollary \ref{rank1}.
However, there may be gaps in the range.

\medskip\noindent\emph{Claim.} For any natural number $n$ there are no rank $2$ vertices in $\Hom(\Diamond_n,\Delta_2)$.

\medskip First, for every polytope $P$ and an element $f\in\vertex(\Diamond_n,P)$, one has
$$
\vertex(\Im f)\subset\vertex(P\cap(-P^{(c)})),
$$
where $c$ is the center of $\Im f$ and $-P^{(c)}$ is the symmetric image of $P$ w.r.t. $c$. In fact, because $\Im f$ is centrally symmetric, we have $\Im f\subset P\cap(-P^{(c)})$. But if there is a vertex $v\in \Im f$ not in $\vertex(P\cap(-P^{(c)}))$, then there is an open interval $I\subset\Im f$, containing $v$. Then, by sliding $v$ along $I$ (and $-v^{(c)}$ along $-I^{(c)}$ in the opposite direction), one can define an affine 1-parameter $(f_t)_{(-1,1)}\subset\Hom(\Diamond_n,P)$ with $f_0=f$. Such is not possible in view of $(\pc_1)$.

Returning to the case $P=\Delta_2$, assume to the contrary that $f$ is
a rank $2$ vertex of $\Hom(\Diamond_n,\Delta_2)$. Let $c$ be the
center of $\Im f$. There are two cases: (i)
$\Delta_2\cap(-\Delta_2^{(c)})$ is a parallelogram, or (ii)
$\Delta_2\cap(-\Delta_2^{(c)})$ is a centrally symmetric hexagon. In the
first case, there is a vertex $x\in\Delta_2$ such that $-x^{(c)}$ is in
the interior of the edge $E\subset\Delta_2$, opposite to $x$. So
sliding $c$ along a small open interval $c\in I$, parallel to $E$,
produces an affine family of parallelograms
$(\Delta_2\cap(-\Delta_2^{(c)}))_I$ in $\Delta_2$. The latter can be used
to define an affine 1-parameter family
$(f_t)_{(-1,1)}\subset\Hom(\Diamond_n,\Delta_2)$ with $f_0=f$. In the
second case, we use the similar sliding procedure, except now the small
open interval $I\subset\Delta_2$, containing $c$, is not constrained
to have any particular direction -- we can always define an affine family of centrally symmetric hexagons $(\Delta_2\cap(-\Delta_2^{(c)}))_I$.  In either case we get a contradiction by $(\pc_1)$.
\end{example}

\begin{figure}
\begin{tabular}{cc}
\begin{tikzpicture}[scale=0.7]
\fill(0.1,0) circle(0.13);
\fill(2.9,0) circle(0.13);
\fill(5.4,0) circle(0.13);
\fill(0.25,4) circle(0.13);
\fill(3.2,4) circle(0.13);
\fill(6,4) circle(0.13);
\fill(3.0,1.9) circle(0.13);
\draw (0.1,0) -- (5.4,0);
\draw (5.4,0) -- (3.2,4);
\draw (3.2,4) -- (0.1,0);
\draw (0.25,4) -- (6,4);
\draw (6,4) -- (2.9,0);
\draw (2.9,0) -- (0.25,4);
\draw[<->](2.5,1.9) -- (3.5,1.9);
\draw (3,1.35) node{$c$};
\draw (-0.5,-0.5) node{$\Delta_2$};
\draw (6,3) node{-$\Delta_2^{(c)}$};
\end{tikzpicture}
&
\begin{tikzpicture}[scale=0.7]
\fill(0,0) circle(0.13);
\fill(6.85,0) circle(0.13);
\fill(4.5,4.5) circle(0.13);
\fill(4.2,-1.5) circle(0.13);
\fill(1.7,3) circle(0.13);
\fill(8.2,3) circle(0.13);
\fill(4.3,1.5) circle(0.13);
\draw (0,0) -- (6.85,0);
\draw (6.85,0) -- (4.5,4.5);
\draw (4.5,4.5) -- (0,0);
\draw (1.7,3) -- (4.2,-1.5);
\draw (4.2,-1.5) -- (8.2,3);
\draw (8.2,3) -- (1.7,3);
\draw[<->] (3.8,1.3) -- (4.8,1.7);
\draw (4.3,1.1) node{$c$};
\draw (5.5,4.5) node{$\Delta_2$};
\draw (5.5,-2) node{-$\Delta_2^{(c)}$};
\end{tikzpicture}
\end{tabular}
\caption{The sliding argument in Example~\ref{bisimplexinsimplex}}
\label{fig:sliding}
\end{figure}
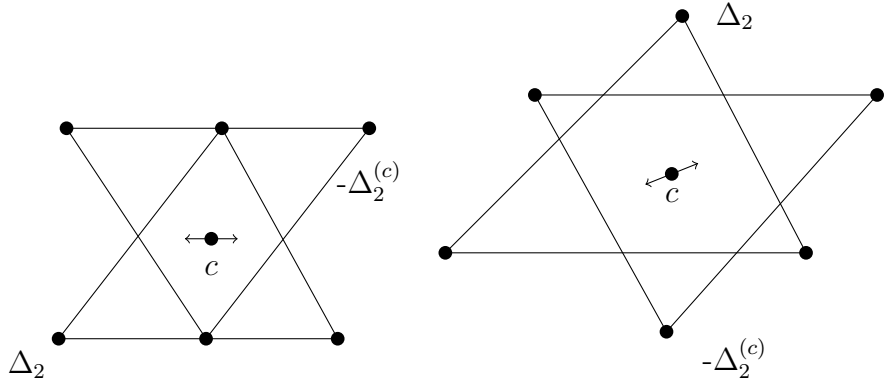

As a side observation, the right picture in Figure~\ref{fig:sliding} gives rise to yet another example of a deflation, followed by an injective vertex map, such that the composite is not a vertex map. In fact, the deflation is a surjective affine map $\Diamond_3\to\Delta_2\cap\Delta_c^{(c)}$ and the injective vertex map is the embedding $\Delta_2\cap\Delta_c^{(c)}\to\Delta_2$. One needs to apply a variant of the sliding argument in Example \ref{nonvertexfactorization} to show that the latter belongs to $\vertex\big(\Delta_2\cap\Delta_c^{(c)},\Delta_2\big)$. The crucial point is that no perturbation of the hexagon in $\Delta_2$ can keep both the isomorphism class of the hexagon and the `hexagon vertex $\leftrightarrow$ triangle edge' incidence table constant.

\section{Generic pairs of polygons}\label{generic}
The goal of this section is to understand some of the structure of the
hom-polytope of a generic pair of polygons $(P,Q)$. The main
result, Theorem~\ref{thm:generic}, is that such a hom-polytope is 'almost
  simple': apart from the short list of vertex maps of rank zero and
  one, every vertex of $\Hom(P,Q)$ is simple.

We begin by constructing spaces of polytopes that will allow
us to precisely state this result. Fix dimensions $d$ and $e$ and
integers $m \geq d+1$, $n \geq e+1$. To every real $d \times m$ matrix
$M$ we associate the polytope $P(M) \subseteq \RR^d$ given as the
convex hull of the columns of $M$. Let $R_{d,m} \subseteq \RR^{d
  \times m}$ be the set of matrices $M$ for which the columns are in
convex position and affinely span $\RR^m$; that is, for which $P(M)$ is a $d$-polytope with $m$
vertices. Similarly, to every real $n \times e$ matrix $M'$ we
associate the polyhedron $Q(M') \subseteq \RR^e$ given by the system of
inequalities $M' \xx \leq \mathbf{1}$, where $\mathbf{1} \in \RR^n$
denotes the vector of all ones. Let $R'_{e,n}$ be
the set of matrices $M'$ for which $Q(M')$ is an $e$-polytope with $n$
facets.

Note that the conditions defining $R_{d,m}$ and $R'_{e,n}$ are stable under small perturbation; that is, $R_{d,m}$ and $R'_{e,n}$
are open subsets of $\RR^{d \times m}$ and $\RR^{n \times e}$, respectively. Also, for generic $M$ and $M'$, $P(M)$ is simplicial and $Q(M')$ is simple. Furthermore, $R_{d,m}\subset\RR^{d\times m}$ and $R'_{e,n}\subset\RR^{n\times e}$ are \emph{semi-algebraic} subsets (i.~e., defined by algebraic equalities and (strict) inequalities). In fact, the convex $d$-polytopes with $m$-vertices give rise to only finitely many combinatorial types and each type is represented by a semi-algebraic subset of $\RR^{d\times m}$, consisting of the matrices whose certain $d\times d$ minors vanish, certain $d\times d$ minors are positive, and certain $d\times d$ minors are negative. Therefore, $R_{d,m}$ is the union of finitely many semi-algebraic sets and, as such, is itself semi-algebraic. A similar argument applies to $R'_{e,n}$.

Clearly, $R_{2,m}$ and $R'_{2,n}$ are the realization spaces of single combinatorial types. Moreover, the realization space of any 3-dimensional combinatorial type is a classically a smooth ball; however, starting from $d=e=4$, the realization space of a combinatorial type can be arbitrarily complicated; see \cite{Richter-Gebert}.

It is in the context above that we consider the hom-polytopes $\Hom(P(M), Q(M'))$ for a generic pair
$(M,M')$.

We now focus on the case of pairs of polygons: $d=e=2$.
Every $m$-gon in the plane is associated to a matrix in $R_{2,m}$
that is unique up to permuting its columns. Every $n$-gon in the plane
that contains the origin in its interior is associated to a matrix in $R'_{2,n}$ that is
unique up to permuting its rows. So $R_{2,m} \times R'_{2,n}$,
a full-dimensional semi-algebraic open subset of the Euclidean space $\RR^{2
  \times m} \times \RR^{n \times 2}$, effectively represents the space
of pairs of polygons.

One the one hand, we have $\dim(R_{2,m} \times R'_{2,n})=2m+2n$. On the other hand, for every pair $(M,M')\in \dim(R_{2,m} \times R'_{2,n})$,
the space of all small perturbations of the facets of $\Hom(P(M),Q(M'))$, keeping the facet-normals invariant, has dimension $mn$ (by Proposition \ref{homproperties}(1)).
Yet, we have

\begin{theorem} \label{thm:generic}
There is a dense open subset $U_{m,n}$ of $R_{2,m} \times
R'_{2,n}$ such that if $(M,M') \in U_{m,n}$, then every vertex map
$f:P(M) \to Q(M')$ of rank two is a simple vertex of $\Hom(P(M), Q(M'))$.
\end{theorem}

\begin{proof}
Write $P:=P(M)$, $Q:=Q(M')$. Since $d=e=2$, $\Hom(P,Q)$ is a six-dimensional polytope and we can write
\[ M =
\begin{pmatrix}
s_1 & \hdots & s_m \\
t_1 & \hdots & t_m
\end{pmatrix},
M' =
\begin{pmatrix}
u_1 & v_1 \\
\vdots & \vdots \\
u_n & v_n
\end{pmatrix}. \]

To verify the conclusion of Theorem~\ref{thm:generic} for fixed $P$
and $Q$, it suffices to show the stronger statement that if a map $f \in
\aff(\RR^2,\RR^2)$ lies in the intersection of the affine hulls of
seven facets of $\Hom(P,Q)$, then either
\begin{itemize}
\item $f(P)$ is not contained in $Q$, or
\item $f$ is not of full rank.
\end{itemize}

For each $i=1,\dots,n$ and $j=1,\dots,m$
there is a facet of $\Hom(P,Q)$ whose affine span is given by the
equation $$(u_i, v_i) \cdot f(s_j,t_j) = 1.$$

If we write $f(x,y) = (\alpha x+ \beta y+\gamma, \delta x+ \epsilon y+
\zeta)$, then more explicitly,
this affine span is given by
\[ \begin{pmatrix} u_i & v_i & 0 \end{pmatrix}
\begin{pmatrix} \alpha & \beta & \gamma \\
\delta & \epsilon & \zeta \\
0 & 0 & 1\end{pmatrix}
\begin{pmatrix} s_j \\ t_j \\ 1 \end{pmatrix} = 1, \] \label{eq:facetmatrix}
or as a linear constraint on the entries of $f$,
\[ (u_i s_j) \alpha + (u_i t_j) \beta + u_i \gamma + (v_is_j)\delta +
(v_it_j) \epsilon + v_i \zeta
= 1. \] \label{eq:facetlinear}
In other words, the condition for the affine hulls of seven facets to meet is that the last column in the matrix below is a linear combination of the first six columns:
\[ A_{i_1,\dots,i_7,j_1,\dots,j_7}=\begin{pmatrix}
u_{i_1}s_{j_1} & u_{i_1}t_{j_1} & u_{i_1} & v_{i_1}s_{j_1} &
v_{i_1}t_{j_1} & v_{i_1} & -1 \\
u_{i_2}s_{j_2} & u_{i_2}t_{j_2} & u_{i_2} & v_{i_2}s_{j_2} &
v_{i_2}t_{j_2} & v_{i_2} & -1 \\
\vdots & \vdots & \vdots & \vdots & \vdots & \vdots & \vdots \\
u_{i_7}s_{j_7} & u_{i_7}t_{j_7} & u_{i_7} & v_{i_7}s_{j_7} &
v_{i_7}t_{j_7} & v_{i_7} & -1 \end{pmatrix} \in \RR^{7 \times 7}
\]
So the matrix $A_{i_1,\dots,i_7,j_1,\dots,j_7}$ needs to be singular for some distinct pairs of indices $(i_1,j_1) \dots,
(i_7,j_7)$. Note that the indices $i_1, \dots i_7$ need not all be
distinct, nor do $j_1, \dots j_7$.

Thus it will suffice to show that for all choices of $(i_1,j_1), \dots,
(i_7,j_7)$, either
\begin{itemize}
\item the seven facets do not intersect in any point of
  $\Hom(P,Q)$, or
\item every map in the intersection of the seven facets is of less
  than full rank, or
\item the generic matrix
$A_{i_1,\dots,i_7,j_1,\dots,j_7}$ is nonsingular.
\end{itemize}

We index the various cases for lists of indices by \emph{coincidence graphs}: bipartite
graphs $G = (A,B,E)$ with exactly seven edges and no isolated
nodes. The set $A$ represents the distinct elements of the list $i_1,
\dots i_7$ (facets of $Q$) and $B$ represents the distinct elements
of $j_1, \dots, j_7$ (vertices of $P$.) The nodes of $G$ are not
labelled but the two parts $A$ and $B$ of $G$ are distinguishable: the condition
that two vertices of $P$ land on the same edge line (i.e. affine span
of an edge) of $Q$ is not the same as the condition that a vertex of
$P$ lands on the intersection of two different edge lines of $Q$. For
instance, the case that $i_1,\ldots,i_7$ are all distinct and
$j_1,\ldots,j_7$ are also all distinct is encoded by the graph with
seven vertex-disjoint edges, the first graph in Table~\ref{tab:allgraphs}. The case that $j_1=j_2,i_3=i_4$, and
everything else is distinct is encoded by the third graph in the left column of Table~\ref{tab:allgraphs}.

\begin{observation} \label{obs:whichgraphs}
$ $
\begin{enumerate}
\item If $A$ contains a node of degree greater than two, then there is a vertex
  $v$ of $P$ such that $f(v)$ is a point of intersection of three
  distinct edge lines of $Q$. But no three edge lines of a polygon can meet at a point (whether this point is inside or outside of the polygon).
\item If $B$ contains a node of degree greater than two,
  then $f$ sends three different vertices of $P$ onto the same edge
  line of $Q$. Since any three vertices of $P$ are affinely
  independent, this implies that $f$ sends all of $\RR^2$ onto the
  same line; i.e, $f$ is not of full rank.
\item If $G$ contains a 4-cycle, then $f$ sends two vertices of $P$
  to the same point: the intersection of two edge lines of $Q$. So
  again $f$ is not of full rank.
\item If $G$ contains a 6-cycle, then there are vertices $v_1,
  v_2, v_3$ of $P$ and edge lines $\ell_1,
  \ell_2, \ell_3$ of $Q$ such that $f(v_1) = \ell_1 \cap
  \ell_2$, $f(v_2) = \ell_2 \cap \ell_3$, and $f(v_3) =
  \ell_3 \cap \ell_1$. For the condition $f(P) \subseteq Q$
  to also be satisfied, the three intersection points $\ell_1 \cap \ell_2$, $\ell_2 \cap \ell_3$,
  and $\ell_3 \cap \ell_1$ must all be vertices of $Q$. That is, $Q$
  is a triangle. Then since $f(P)$ contains all three vertices of $Q$,
  $f(P) = Q$, which implies that $P$ is also a triangle. Then
  $\Hom(P,Q) = Q^3$, the product of three triangles, and then all of its
  vertices are simple.
\end{enumerate}
\end{observation}

In summary, we may now assume that $G$ has no cycles and no vertices of
degree greater than two. That is, $G$ is a union of
vertex-disjoint paths with exactly seven edges. The 31 such graphs
(with distinguished upper part $A$ and lower part $B$) are shown
in Table~\ref{tab:allgraphs}.

\medskip
\begin{table}
\begin{tabular}{c|c|c|c}
\begin{tikzpicture}[scale=0.5]
\draw (0,0) -- (0,1.5);
\draw (1,0) -- (1,1.5);
\draw (2,0) -- (2,1.5);
\draw (3,0) -- (3,1.5);
\draw (4,0) -- (4,1.5);
\draw (5,0) -- (5,1.5);
\draw (6,0) -- (6,1.5);
\foreach \i in {0,...,6}
{
	\fill (\i, 0) circle (0.2);
}
\foreach \i in {0,...,6}
{
        \fill (\i, 1.5) circle (0.2);
}
\end{tikzpicture} &

\begin{tikzpicture}[scale=0.5]
\draw (0,0) -- (0,1.5);
\draw (0,0) -- (1,1.5);
\draw (2,0) -- (2,1.5);
\draw (3,0) -- (3,1.5);
\draw (4,0) -- (4,1.5);
\draw (5,0) -- (5,1.5);
\draw (6,0) -- (6,1.5);
\fill(0,0) circle (0.2);
\foreach \i in {2,...,6}
{
	\fill (\i, 0) circle (0.2);
}
\foreach \i in {0,...,6}
{
        \fill (\i, 1.5) circle (0.2);
}
\end{tikzpicture} &

\begin{tikzpicture}[scale=0.5]
\draw (0,0) -- (0,1.5);
\draw (1,0) -- (0,1.5);
\draw (2,0) -- (2,1.5);
\draw (3,0) -- (3,1.5);
\draw (4,0) -- (4,1.5);
\draw (5,0) -- (5,1.5);
\draw (6,0) -- (6,1.5);
\foreach \i in {0,...,6}
{
	\fill (\i, 0) circle (0.2);
}
\fill(0,1.5) circle (0.2);
\foreach \i in {2,...,6}
{
        \fill (\i, 1.5) circle (0.2);
}
\end{tikzpicture}
&
\begin{tikzpicture}[scale=0.5]
\draw (0,0) -- (0,1.5);
\draw (0,0) -- (1,1.5);
\draw (1,0) -- (1,1.5);
\draw (2,0) -- (2,1.5);
\draw (3,0) -- (3,1.5);
\draw (4,0) -- (4,1.5);
\draw (5,0) -- (5,1.5);
\foreach \i in {0,...,5}
{
	\fill (\i, 0) circle (0.2);
}
\foreach \i in {0,...,5}
{
        \fill (\i, 1.5) circle (0.2);
}
\end{tikzpicture}
\\[8pt]

\begin{tikzpicture}[scale=0.5]
\draw (0,0) -- (0,1.5);
\draw (0,0) -- (1,1.5);
\draw (2,0) -- (2,1.5);
\draw (2,0) -- (3,1.5);
\draw (4,0) -- (4,1.5);
\draw (5,0) -- (5,1.5);
\draw (6,0) -- (6,1.5);
\fill (0,0) circle (0.2);
\fill (2,0) circle (0.2);
\foreach \i in {4,...,6}
{
	\fill (\i, 0) circle (0.2);
}
\foreach \i in {0,...,6}
{
        \fill (\i, 1.5) circle (0.2);
}
\end{tikzpicture} &

\begin{tikzpicture}[scale=0.5]
\draw (0,0) -- (0,1.5);
\draw (1,0) -- (0,1.5);
\draw (2,0) -- (2,1.5);
\draw (3,0) -- (2,1.5);
\draw (4,0) -- (4,1.5);
\draw (5,0) -- (5,1.5);
\draw (6,0) -- (6,1.5);
\foreach \i in {0,...,6}
{
	\fill (\i, 0) circle (0.2);
}
\fill (0,1.5) circle (0.2);
\fill (2,1.5) circle (0.2);
\foreach \i in {4,...,6}
{
        \fill (\i, 1.5) circle (0.2);
}
\end{tikzpicture}
&
\begin{tikzpicture}[scale=0.5]
\draw (0,0) -- (0,1.5);
\draw (0,0) -- (1,1.5);
\draw (2,0) -- (2,1.5);
\draw (3,0) -- (2,1.5);
\draw (4,0) -- (4,1.5);
\draw (5,0) -- (5,1.5);
\draw (6,0) -- (6,1.5);
\fill (0,0) circle (0.2);
\foreach \i in {2,...,6}
{
	\fill (\i, 0) circle (0.2);
}
\fill (0,1.5) circle (0.2);
\fill (1,1.5) circle (0.2);
\fill (2,1.5) circle (0.2);
\foreach \i in {4,...,6}
{
        \fill (\i, 1.5) circle (0.2);
}
\end{tikzpicture}
&
\begin{tikzpicture}[scale=0.5]
\draw (0,0) -- (0,1.5);
\draw (0,0) -- (1,1.5);
\draw (1,0) -- (1,1.5);
\draw (1,0) -- (2,1.5);
\draw (3,0) -- (3,1.5);
\draw (4,0) -- (4,1.5);
\draw (5,0) -- (5,1.5);

\fill (0,0) circle (0.2);
\fill (1,0) circle (0.2);
\foreach \i in {3,...,5}
{
	\fill (\i, 0) circle (0.2);
}
\foreach \i in {0,...,5}
{
        \fill (\i, 1.5) circle (0.2);
}
\end{tikzpicture}
\\[8pt]
\begin{tikzpicture}[scale=0.5]
\draw (0,0) -- (0,1.5);
\draw (1,0) -- (0,1.5);
\draw (1,0) -- (1,1.5);
\draw (2,0) -- (1,1.5);
\draw (3,0) -- (3,1.5);
\draw (4,0) -- (4,1.5);
\draw (5,0) -- (5,1.5);
\foreach \i in {0,...,5}
{
	\fill (\i, 0) circle (0.2);
}
\fill (0,1.5) circle (0.2);
\fill (1,1.5) circle (0.2);
\foreach \i in {3,...,5}
{
        \fill (\i, 1.5) circle (0.2);
}
\end{tikzpicture}
&
\begin{tikzpicture}[scale=0.5]
\draw (0,0) -- (0,1.5);
\draw (0,0) -- (1,1.5);
\draw (1,0) -- (1,1.5);
\draw (2,0) -- (2,1.5);
\draw (2,0) -- (3,1.5);
\draw (4,0) -- (4,1.5);
\draw (5,0) -- (5,1.5);
\foreach \i in {0,...,2}
{
	\fill (\i, 0) circle (0.2);
}
\foreach \i in {4,...,5}
{
	\fill (\i, 0) circle (0.2);
}
\foreach \i in {0,...,5}
{
        \fill (\i, 1.5) circle (0.2);
}
\end{tikzpicture}
&
\begin{tikzpicture}[scale=0.5]
\draw (0,0) -- (0,1.5);
\draw (0,0) -- (1,1.5);
\draw (1,0) -- (1,1.5);
\draw (2,0) -- (2,1.5);
\draw (3,0) -- (2,1.5);
\draw (4,0) -- (4,1.5);
\draw (5,0) -- (5,1.5);
\foreach \i in {0,...,5}
{
	\fill (\i, 0) circle (0.2);
}
\foreach \i in {0,...,2}
{
        \fill (\i, 1.5) circle (0.2);
}
\foreach \i in {4,...,5}
{
        \fill (\i, 1.5) circle (0.2);
}
\end{tikzpicture}
&
\begin{tikzpicture}[scale=0.5]
\draw (0,0) -- (0,1.5);
\draw (0,0) -- (1,1.5);
\draw (2,0) -- (2,1.5);
\draw (2,0) -- (3,1.5);
\draw (4,0) -- (4,1.5);
\draw (4,0) -- (5,1.5);
\draw (6,0) -- (6,1.5);
\fill (0, 0) circle (0.2);
\fill (2, 0) circle (0.2);
\fill (4, 0) circle (0.2);
\fill (6, 0) circle (0.2);
\foreach \i in {0,...,6}
{
        \fill (\i, 1.5) circle (0.2);
}
\end{tikzpicture}
\\[8pt]
\begin{tikzpicture}[scale=0.5]
\draw (0,0) -- (0,1.5);
\draw (0,0) -- (1,1.5);
\draw (2,0) -- (2,1.5);
\draw (2,0) -- (3,1.5);
\draw (4,0) -- (4,1.5);
\draw (5,0) -- (4,1.5);
\draw (6,0) -- (6,1.5);
\fill (0, 0) circle (0.2);
\fill (2, 0) circle (0.2);
\foreach \i in {4,...,6}
{
	\fill (\i, 0) circle (0.2);
}
\foreach \i in {0,...,4}
{
        \fill (\i, 1.5) circle (0.2);
}
\fill (6, 1.5) circle (0.2);
\end{tikzpicture}
&
\begin{tikzpicture}[scale=0.5]
\draw (0,0) -- (0,1.5);
\draw (0,0) -- (1,1.5);
\draw (2,0) -- (2,1.5);
\draw (3,0) -- (2,1.5);
\draw (4,0) -- (4,1.5);
\draw (5,0) -- (4,1.5);
\draw (6,0) -- (6,1.5);
\fill (0, 0) circle (0.2);
\foreach \i in {2,...,6}
{
	\fill (\i, 0) circle (0.2);
}
\foreach \i in {0,...,2}
{
        \fill (\i, 1.5) circle (0.2);
}
\fill (4, 1.5) circle (0.2);
\fill (6, 1.5) circle (0.2);
\end{tikzpicture}
&
\begin{tikzpicture}[scale=0.5]
\draw (0,0) -- (0,1.5);
\draw (1,0) -- (0,1.5);
\draw (2,0) -- (2,1.5);
\draw (3,0) -- (2,1.5);
\draw (4,0) -- (4,1.5);
\draw (5,0) -- (4,1.5);
\draw (6,0) -- (6,1.5);
\foreach \i in {0,...,6}
{
	\fill (\i, 0) circle (0.2);
}
\fill (0, 1.5) circle (0.2);
\fill (2, 1.5) circle (0.2);
\fill (4, 1.5) circle (0.2);
\fill (6, 1.5) circle (0.2);
\end{tikzpicture}
&


\begin{tikzpicture}[scale=0.5]
\draw (0,0) -- (0,1.5);
\draw (0,0) -- (1,1.5);
\draw (1,0) -- (1,1.5);
\draw (1,0) -- (2,1.5);
\draw (2,0) -- (2,1.5);
\draw (3,0) -- (3,1.5);
\draw (4,0) -- (4,1.5);
\foreach \i in {0,...,4}
{
	\fill (\i, 0) circle (0.2);
}
\foreach \i in {0,...,4}
{
        \fill (\i, 1.5) circle (0.2);
}
\end{tikzpicture}
\\[8pt]

\begin{tikzpicture}[scale=0.5]
\draw (0,0) -- (0,1.5);
\draw (0,0) -- (1,1.5);
\draw (1,0) -- (1,1.5);
\draw (1,0) -- (2,1.5);
\draw (3,0) -- (3,1.5);
\draw (3,0) -- (4,1.5);
\draw (5,0) -- (5,1.5);
\fill(0,0) circle (0.2);
\fill(1,0) circle (0.2);
\fill(3,0) circle (0.2);
\fill(5,0) circle (0.2);
\foreach \i in {0,...,5}
{
        \fill (\i, 1.5) circle (0.2);
}
\end{tikzpicture}
&
\begin{tikzpicture}[scale=0.5]
\draw (0,0) -- (0,1.5);
\draw (0,0) -- (1,1.5);
\draw (1,0) -- (1,1.5);
\draw (1,0) -- (2,1.5);
\draw (3,0) -- (3,1.5);
\draw (4,0) -- (3,1.5);
\draw (5,0) -- (5,1.5);
\fill(0,0) circle (0.2);
\fill(1,0) circle (0.2);
\foreach \i in {3,...,5}
{
	\fill (\i, 0) circle (0.2);
}
\foreach \i in {0,...,3}
{
        \fill (\i, 1.5) circle (0.2);
}
\fill(5,1.5) circle (0.2);
\end{tikzpicture}
&

\begin{tikzpicture}[scale=0.5]
\draw (0,0) -- (0,1.5);
\draw (1,0) -- (0,1.5);
\draw (1,0) -- (1,1.5);
\draw (2,0) -- (1,1.5);
\draw (3,0) -- (3,1.5);
\draw (3,0) -- (4,1.5);
\draw (5,0) -- (5,1.5);
\foreach \i in {0,...,3}
{
	\fill (\i, 0) circle (0.2);
}
\fill(5,0) circle (0.2);
\fill(0,1.5) circle (0.2);
\fill(1,1.5) circle (0.2);
\foreach \i in {3,...,5}
{
        \fill (\i, 1.5) circle (0.2);
}
\end{tikzpicture}
&
\begin{tikzpicture}[scale=0.5]
\draw (0,0) -- (0,1.5);
\draw (1,0) -- (0,1.5);
\draw (1,0) -- (1,1.5);
\draw (2,0) -- (1,1.5);
\draw (3,0) -- (3,1.5);
\draw (4,0) -- (3,1.5);
\draw (5,0) -- (5,1.5);
\foreach \i in {0,...,5}
{
	\fill (\i, 0) circle (0.2);
}
\fill(0,1.5) circle (0.2);
\fill(1,1.5) circle (0.2);
\fill(3,1.5) circle (0.2);
\fill(5,1.5) circle (0.2);
\end{tikzpicture}
\\[8pt]
\begin{tikzpicture}[scale=0.5]
\draw (0,0) -- (0,1.5);
\draw (0,0) -- (1,1.5);
\draw (1,0) -- (1,1.5);
\draw (2,0) -- (2,1.5);
\draw (2,0) -- (3,1.5);
\draw (3,0) -- (3,1.5);
\draw (4,0) -- (4,1.5);
\foreach \i in {0,...,4}
{
	\fill (\i, 0) circle (0.2);
}
\foreach \i in {0,...,4}
{
        \fill (\i, 1.5) circle (0.2);
}
\end{tikzpicture}
&
\begin{tikzpicture}[scale=0.5]
\draw (0,0) -- (0,1.5);
\draw (0,0) -- (1,1.5);
\draw (1,0) -- (1,1.5);
\draw (2,0) -- (2,1.5);
\draw (2,0) -- (3,1.5);
\draw (4,0) -- (4,1.5);
\draw (5,0) -- (5,1.5);
\foreach \i in {0,...,2}
{
	\fill (\i, 0) circle (0.2);
}
\fill (4,0) circle (0.2);
\fill (5,0) circle (0.2);
\foreach \i in {0,...,5}
{
        \fill (\i, 1.5) circle (0.2);
}
\end{tikzpicture}
&
\begin{tikzpicture}[scale=0.5]
\draw (0,0) -- (0,1.5);
\draw (0,0) -- (1,1.5);
\draw (1,0) -- (1,1.5);
\draw (2,0) -- (2,1.5);
\draw (2,0) -- (3,1.5);
\draw (4,0) -- (4,1.5);
\draw (5,0) -- (4,1.5);
\foreach \i in {0,...,2}
{
	\fill (\i, 0) circle (0.2);
}
\fill (4,0) circle (0.2);
\fill (5,0) circle (0.2);
\foreach \i in {0,...,4}
{
        \fill (\i, 1.5) circle (0.2);
}
\end{tikzpicture}
&
\begin{tikzpicture}[scale=0.5]
\draw (0,0) -- (0,1.5);
\draw (0,0) -- (1,1.5);
\draw (1,0) -- (1,1.5);
\draw (2,0) -- (2,1.5);
\draw (3,0) -- (2,1.5);
\draw (4,0) -- (4,1.5);
\draw (5,0) -- (4,1.5);
\foreach \i in {0,...,5}
{
	\fill (\i, 0) circle (0.2);
}
\foreach \i in {0,...,2}
{
        \fill (\i, 1.5) circle (0.2);
}
\fill (4,1.5) circle (0.2);
\end{tikzpicture}
\\[8pt]
\begin{tikzpicture}[scale=0.5]
\draw (0,0) -- (0,1.5);
\draw (0,0) -- (1,1.5);
\draw (1,0) -- (1,1.5);
\draw (1,0) -- (2,1.5);
\draw (2,0) -- (2,1.5);
\draw (2,0) -- (3,1.5);
\draw (4,0) -- (4,1.5);
\foreach \i in {0,...,2}
{
	\fill (\i, 0) circle (0.2);
}
\fill (4,0) circle (0.2);
\foreach \i in {0,...,4}
{
        \fill (\i, 1.5) circle (0.2);
}
\end{tikzpicture}
&
\begin{tikzpicture}[scale=0.5]
\draw (0,0) -- (0,1.5);
\draw (1,0) -- (0,1.5);
\draw (1,0) -- (1,1.5);
\draw (2,0) -- (1,1.5);
\draw (2,0) -- (2,1.5);
\draw (3,0) -- (2,1.5);
\draw (4,0) -- (4,1.5);
\foreach \i in {0,...,4}
{
	\fill (\i, 0) circle (0.2);
}
\foreach \i in {0,...,2}
{
        \fill (\i, 1.5) circle (0.2);
}
\fill (4,1.5) circle (0.2);
\end{tikzpicture}
&
\begin{tikzpicture}[scale=0.5]
\draw (0,0) -- (0,1.5);
\draw (0,0) -- (1,1.5);
\draw (1,0) -- (1,1.5);
\draw (1,0) -- (2,1.5);
\draw (2,0) -- (2,1.5);
\draw (3,0) -- (3,1.5);
\draw (3,0) -- (4,1.5);
\foreach \i in {0,...,3}
{
	\fill (\i, 0) circle (0.2);
}
\foreach \i in {0,...,4}
{
        \fill (\i, 1.5) circle (0.2);
}
\end{tikzpicture}
&
\begin{tikzpicture}[scale=0.5]
\draw (0,0) -- (0,1.5);
\draw (0,0) -- (1,1.5);
\draw (1,0) -- (1,1.5);
\draw (1,0) -- (2,1.5);
\draw (2,0) -- (2,1.5);
\draw (3,0) -- (3,1.5);
\draw (4,0) -- (3,1.5);
\foreach \i in {0,...,4}
{
	\fill (\i, 0) circle (0.2);
}
\foreach \i in {0,...,3}
{
        \fill (\i, 1.5) circle (0.2);
}
\end{tikzpicture}
\\[8pt]

\begin{tikzpicture}[scale=0.5]
\draw (0,0) -- (0,1.5);
\draw (0,0) -- (1,1.5);
\draw (1,0) -- (1,1.5);
\draw (1,0) -- (2,1.5);
\draw (3,0) -- (3,1.5);
\draw (3,0) -- (4,1.5);
\draw (4,0) -- (4,1.5);
\foreach \i in {0,...,1}
{
	\fill (\i, 0) circle (0.2);
}
\foreach \i in {3,...,4}
{
	\fill (\i, 0) circle (0.2);
}
\foreach \i in {0,...,4}
{
        \fill (\i, 1.5) circle (0.2);
}
\end{tikzpicture}
&
\begin{tikzpicture}[scale=0.5]
\draw (0,0) -- (0,1.5);
\draw (1,0) -- (0,1.5);
\draw (1,0) -- (1,1.5);
\draw (2,0) -- (1,1.5);
\draw (3,0) -- (3,1.5);
\draw (3,0) -- (4,1.5);
\draw (4,0) -- (4,1.5);
\foreach \i in {0,...,4}
{
	\fill (\i, 0) circle (0.2);
}
\foreach \i in {0,...,1}
{
        \fill (\i, 1.5) circle (0.2);
}
\foreach \i in {3,...,4}
{
        \fill (\i, 1.5) circle (0.2);
}
\end{tikzpicture}
&
\begin{tikzpicture}[scale=0.5]
\draw (0,0) -- (0,1.5);
\draw (0,0) -- (1,1.5);
\draw (1,0) -- (1,1.5);
\draw (1,0) -- (2,1.5);
\draw (2,0) -- (2,1.5);
\draw (2,0) -- (3,1.5);
\draw (3,0) -- (3,1.5);
\foreach \i in {0,...,3}
{
	\fill (\i, 0) circle (0.2);
}
\foreach \i in {0,...,3}
{
        \fill (\i, 1.5) circle (0.2);
}
\end{tikzpicture}
&
\end{tabular}

\bigskip\caption{The 31 possible coincidence graphs}
\label{tab:allgraphs}
\end{table}

For each graph $G$, we used the computer algebra system
\texttt{Macaulay 2} \cite{M2} to
compute the appropriate generic determinant $D_G$
and to verify that it is not identically zero.
It follows that in each case, $D_G$ is nonzero on a dense open subset
of  $\RR^{m \times 2} \times \RR^{2 \times n}$, the complement of an
algebraic hypersurface. Since
$R_{2,m}$ and $R_{2,n}'$ are full-dimensional subsets of $\RR^{m \times 2}$ and $\RR^{2
  \times n}$, we conclude that $D_G$ is also nonzero on a dense open
subset $U_G$ of $R_{2,m} \times R_{2,n}'$. We now let $U_{n,m} := \bigcap_{G}U_G$,
where the intersection is taken over the 31 coincidence graphs
$G$. For every pair $(M,M')$ in $U_{n,m}$,
every full-rank vertex of $\Hom(P,Q)$ is simple.
\end{proof}

\begin{corollary}
For every $m,n\ge3$ there is a dense open subset $V_{m,n}$ of $R_{2,m} \times R_{2,n}'$ such
that for $n$ approaching infinity and fixed $m$, the ratio of simple
vertices to all vertices in $\Hom(P,Q)$ tends to one under that
condition $(M,M') \in V_{m,n}$.
\end{corollary}

\begin{proof}
We take $V_{m,n}$ to be the dense open subset of $U_{m,n}$ given by
the additional constraint that $P$ has no parallel edges. For fixed $(M,M') \in V_{m,n}$, we
group the vertex maps $f \in \Hom(P,Q)$ into four classes as follows.
\begin{enumerate}
\item $f$ has rank zero; that is, $f$ maps all of $P$ to a
  single vertex of $Q$. Then $f$ is not a simple vertex unless $m=3$.
\item $f$ has rank one and $\textup{im}(f)$ is an edge of
  $Q$. Again $f$ is not simple unless $m=3$.
\item $f$ has rank one and $\textup{im}(f)$ is a proper diagonal
  of $Q$. Then by Corollary~\ref{rank1}(1), $f$ is a face collapse;
  that is, it maps an edge of $P$ to one end of the diagonal and
  (since $P$ has no parallel edges) a vertex to the other end. All other vertices of $P$ are mapped into the interior of $Q$. This means $f$ satisfies exactly six facet inequalities of
  $\Hom(P,Q)$, so it is simple.
\item $f$ has rank two. Then $f$ is simple by
  Theorem~\ref{thm:generic}.
\end{enumerate}

By an easy refinement of Corollary~\ref{rank1}(4), we see that there are
$2mn$ vertex maps of type (2), and $mn(n-2)$ of type (3). There are
also $n$ of type (1). Let $k$ be the number of vertex maps of type
(4), which may depend on the particular choice of $(M,M') \in
V_{m,n}$. Regardless of this choice, the ratio of simple
to total vertex maps is (if $m>3$) at least
$$ \frac{mn(n-2)+k}{n+2mn+mn(n-2)+k} \geq \frac{mn(n-2)}{n+mn^2} =
\frac{n-2}{m^{-1}+n}, $$
which tends to one as $n$ tends to infinity.
\end{proof}


We now return to the case of arbitrary dimension, where we make the
following conjecture which would generalize
Theorem~\ref{thm:generic}.

\begin{conjecture}
For any dimensions $d$ and $e$ and any integers $m \geq d+1$ and $n
\geq e+1$, there is a dense open subset $U_{d,e,m,n}$ of $R_{d,m}
\times R'_{e,n}$ such that if $(M,M') \in U_{d,e,m,n}$, then every
vertex map $f:P(M) \to Q(M')$ of full rank is a simple vertex of
$\Hom(P,Q)$.
\end{conjecture}

The motivation is that the main idea of the proof of
Theorem~\ref{thm:generic}, which is to use a finite list of graphs to
enumerate the situations under which a vertex of $\Hom(P(M),Q(M)$ might
\emph{not} be simple and then to show that each is equivalent to the
nonvanishing of a certain generic determinant, does not appear to depend
on the dimensions of the source and target polytopes. However, some of
the conditions in Observation~\ref{obs:whichgraphs} do not apply in
the more general setting, and the number of graphs that must be
checked grows very rapidly with $d$ and $e$.




\section{Regular polygons}\label{regular}
We now consider the hom-polytopes of pairs $(P_m,P_n)$ of regular
polygons. There do not appear to be general descriptions of all
vertices of $\Hom(P_m,P_n)$, or even of $\Hom(P_n,P_n)$, so we only
consider special cases. Again the main goal is to understand
the set of vertices. In fact we need only consider those of full rank,
since we have the following immediate consequence of
Corollary~\ref{rank1}(4).

\begin{corollary} \label{cor:rank1regular}
The number of rank zero vertices of $\Hom(P_m,P_n)$ is $n$ and the
number of rank one vertices is
$$ \left\{
	\begin{array}{ll}
		mn(n-1) & \textup{ for } m \textup{ odd, } \\
	\frac{mn(n-1)}{2} & \textup{ for } m \textup{ even. } \\
	\end{array}
\right. $$
\end{corollary}

For any positive integer $k \geq 3$, let $\ZZ_k$ be the cyclic group of
order $k$ and let $D_k$ be the dihedral group of order $2k$. Then
$D_k$ acts naturally on $P_k$ by rotation and reflection, with $\ZZ_k$
identified with the subgroup of rotations inside $D_k$. In particular, this induces an
action of $D_m \times D_n$ on $\Hom(P_m, P_n)$. The action respects
faces of each dimension. The action of the subgroup $\ZZ_m \times \ZZ_n$
is free and transitive on the $mn$ facets of $\Hom(P_m,P_n)$. Although
the action on the vertices is more subtle, it is still useful in
enumerating them. We begin with a sample result that holds for
all $m$ and $n$ and serves as a mild check on the experimental data we
will later present.

\begin{proposition} \label{prop:divisibility} Fix $m$ and $n$ and let
$V$ be the number of vertices of $\Hom(P_m,P_n)$. Then
\begin{enumerate}
\item $n$ divides $V$,
\item $m$ divides $V-n$, and
\item $n$ and $V$ have the same parity.
\end{enumerate}
\end{proposition}

\begin{proof} All of these statements follow from applying the
orbit-stabilizer theorem to the action of various subgroups of $D_m
\times D_n$ on $\vertex(\Hom(P_m,P_n))$.

\medskip\noindent(1) The action of any nontrivial rotation in $\ZZ_n$ on
  $\Hom(P_m,P_n)$ fixes only one affine map: the zero map. The zero map is not a vertex
of $\Hom(P_m,P_n)$, so the action is free on the vertices.

\medskip\noindent(2) The action of any nontrivial rotation in $\ZZ_m$ on
$\Hom(P_m,P_n)$ fixes only vertex maps of rank zero.  Thus $m$ divides $V-n$.

\medskip\noindent(3) If $n$ is even, this is immediate from (1).

If $n$ is odd, then a reflection $s \in D_n$ fixes only the
line segment from one vertex of $P_n$ to the midpoint of the opposite
edge. Thus the only maps in $\Hom(P_m, P_n)$ that are fixed by $s$ are
the maps that take $P_m$ to that segment. But then, by the perturbation criterion $(\pc_1)$, the only vertex of $\Hom(P_m,P_n))$ that is fixed by $s$ is the map that takes all of $P_m$ to the mentioned vertex of $P_n$. It follows that $2$ divides $V-1$, so $V$ is odd.
\end{proof}

\begin{proposition}
For any $n \geq 3$, $\Hom(P_3,P_n) = P_n^3$. In particular, the number
of full-rank vertices is $n(n-1)(n-2)$.
\end{proposition}

\begin{proof}
The first statement is a special case of
Corollary~\ref{explicitcomputations}(1). For the second, note that
a map $f$ is a full rank vertex of $\Hom(P_3,P_n)$ if and only if
it sends each vertex of $P_3$ to a distinct vertex of $P_n$.
\end{proof}

\begin{proposition}
The number of full rank vertices of $\Hom(P_n, P_3)$ is
$$ \left\{
	\begin{array}{ll}
		\frac{n(n+1)(n-1)}{4} & \textup{ for } n \textup{ odd } \\
	\frac{n(n-2)(n-4)}{4} & \textup{ for } n \textup{ even } \\
	\end{array}
\right. $$
\end{proposition}

\begin{proof}
Suppose $f$ is a rank two vertex of $\Hom(P, P_3)$ for any polygon
$P$. Then $f$ cannot send more than two vertices to any edge of
$P_3$. Since it lies at the intersection of at least six different
facets of $\Hom(P,P_3)$, it must send exactly two vertices of $P$ to
each edge of $P_3$. That is, it inscribes an affine image of $P$ into
the equilateral $P_3$. But since all triangles in the plane are
affinely isomorphic, this is equivalent to circumscribing an arbitrary
triangle around $P$ itself.

Now specialize to $P=P_n$ and let its edges be $E_0,E_1,\dots,E_{n-1}$
in consecutive order. The condition that a triangle can be drawn
around $P_n$ along three edges $E_i, E_j, E_k$ is that the gap between
each pair of indices, taken cyclically, is at most $\lfloor
(n-1)/2 \rfloor.$ We count such triples separately in the cases where
$n$ is odd and where it is even.

Suppose $n=2p+1$ is odd. Consider just the triples $0=i<j<k$. Then we
must have
$$1 \leq j \leq p,\; k-j \leq p,\; (2p+1)-k \leq p.$$
That is, for $j$ fixed we must have $p+1 \leq k \leq p+j$, giving
$\sum_{j=1}^p j = {p+1 \choose 2}$ triples with $i=0$. Since
$\frac{3}{n}$ of the allowed triples include 0, the total number of triples is
\begin{eqnarray*}\frac{n}{3}{p+1 \choose 2} & = &
  \frac{(p+1)p(2p+1)}{6} \\
& = & \frac{(2p+2)2p(2p+1)}{24} \\
& = & \frac{(n+1)(n-1)n}{24}.
\end{eqnarray*}
Similarly if $n=2p$ is even, the size of each gap must be at most
$p-1$. A similar calculation gives the formula
$\frac{n(n-2)(n-4)}{24}$ for the number of triples of edges in this
case.

(Note that these formulae appear without proof as sequence A060422
in~\cite{Sloane}, listed as the number of triples of \emph{vertices} of
a regular $n$-gon that form acute triangles. By passing to the dual
polygons with respect to 0, it is not hard to see that this is equivalent.)

Finally, once we have chosen the three edges $E_i,E_j,E_k$ of $P_n$,
we can apply an arbitrary symmetry of $P_3$, so we multiply the formulae above by $3! =
6$ to obtain the total number of rank two vertices of $\Hom(P_n,P_3)$.
\end{proof}

\begin{proposition}
The number of vertices of $\Hom(P_m,P_4)$ is
$$ \left\{	\begin{array}{ll}
	  (2m+2)^2 & \textup{ for } m \textup{ odd } \\
	 (m+2)^2 & \textup{ for } m \textup{ even. } \\
	\end{array} \right. $$

In particular the number of full-rank vertices is  $$ \left\{
\begin{array}{ll}
	  4m^2-4m & \textup{ for } m \textup{ odd } \\
	 m^2-2m &  \textup{ for } m \textup{ even } \\
	\end{array} \right. $$
\end{proposition}

\begin{proof}
By Proposition~\ref{homproperties}(3), $\Hom(P_m,P_4) =
  \Hom(P_m,I) \times \Hom(P_m,I)$. Now $\Hom(P_m,I)$ has two vertices
  of rank zero and (by Corollary~\ref{rank1}) has $2m$ vertices of rank
  one if $m$ is odd or $m$ such vertices if $m$ is even.
\end{proof}

Note that if $m$ is even, we obtain an explicit description of
$\Hom(P_m,I)$ from Corollary~\ref{explicitcomputations}(2):
it is a bipyramid over the dual $m$-gon $P_m^0$. The $f$-vector (i.~e., the vector of the numbers of vertices, edges, $2$-faces etc) of this
bipyramid is $(m+2,3m,2m,1)$. Using the general formula
$$ f_i(P \times Q) = \sum_{j=0}^i f_j(P)f_{i-j}(Q),$$
we compute the $f$-vector of $\Hom(P_m,I)$ to be
$$ (m^2+4m+4,6m^2+12m,13m^2+4m,12m^2+2m+4,4m^2+6m,4m,1).$$

\begin{proposition}
The number of full-rank vertices of $\Hom(P_4,P_n)$ is
$$ \left\{	\begin{array}{ll}
	  n^3-9n & \textup{ for } n \textup{ odd } \\
	 n^3-5n^2+6n & \textup{ for } n \textup{ even } \\
	\end{array} \right. $$

\end{proposition}

\begin{proof}
Let $v_0, v_1, v_2, v_3$ and $w_0, w_1,\dots,w_{n-1}$ respectively be
the vertices of $P_4$ and of $P_n$ in counterclockwise
order, with $v_0 = w_0$. If $f:P_4 \to P_n$ is a vertex map
of rank two, then in order to achieve the necessary six incidences of
vertices of $f(P_4)$ with facets of $P_n$,
one of the following must apply:

\begin{enumerate}
\item $f$ sends at least three vertices of $P_4$ to distinct vertices of
$P_n$, or
\item $f$ sends two adjacent vertices of $P_4$ to distinct
  vertices of $P_n$ and the other two onto interior points of edges,
  or
\item $f$ sends two opposite vertices of $P_4$ to distinct
  vertices of $P_n$ and the other two onto interior points of edges.
\end{enumerate}

In case (1), we first observe that an affine map $f:P_4\to P_n$ which maps three vertices of $P_4$ to vertices of $P_n$ is automatically a vertex map; see the comment at the beginning of Example \ref{nonvertexfactorization}.
We may assume up to symmetry that
$f(v_0)=w_0$, $f(v_1) = w_i$, and $f(v_3)=w_{n-j}$ with $0 < i \leq
j$. We now consider different cases for the
interior angle $\Theta$ of the parallelogram $f(P_4)$ at
$w_0$.

If $\Theta$ is acute, then $f(v_2)$ lies outside the unit circle. In
particular it is not in $P_n$, so no vertex maps are obtained this way.

If $\Theta$ is a right angle, then
$f(v_2)$ is a rectangle and we must have that $n=2p$ is even and
$i+j=p$. Furthermore $f(v_2)$ is also a vertex of $P_n$:
specifically it is the vertex $w_p$ directly opposite $w_0$, as in the
first two pictures in Table~\ref{tab:Hom48}. Thus to fix the image of
such a map, we must pick two opposite pairs of vertices of
$P_n$. Finally, taking into account the eight possible orientations of
$f(P_4)$, the number of these maps is $$8{p\choose 2} = 4p(p-1) = n(n-2) = n^2-2n.$$

Finally, if $\Theta$ is obtuse then $f(v_2)$ is necessarily in the
interior of $P_n$ (one uses the symmetry w.r.t. to the perpendicular line through the midpoint of $[w_0,w_{n-j}]$). In particular, $f$ is a valid map from $P_4$ to
$P_n$. The condition for $\Theta$ to be obtuse is that $i+j <
n/2$. This situation holds in the third and fourth pictures in
Table~\ref{tab:Hom47} (with $n=7$ and respectively $i=j=1$;
$i=1,j=2$) and in the third and fourth pictures in
Table~\ref{tab:Hom48} (with $n=8$ and respectively $i=j=1$; $i=1,j=2$.)

By considerations similar to the above, the number of
ways to obtain such a map is
$$ 8(2p+1) {p \choose 2} = n(n-1)(n-3) = n^3-4n^2+3n$$
if $n=2p+1$, or
$$ 8(2p){p-1 \choose 2} = 2p(2p-2)(2p-4) = n(n-2)(n-4) = n^3-6n^2+8n$$
if $n=2p$.

\medskip In case (2), we may assume that $f(v_0) = w_0$ and $f(v_1) =
w_{i}$ for some $i < n/2$. Then, since $P_n$ is symmetric w.r.t. the perpendicular line through the midpoint of $[w_0,w_i]$, we must actually have that $f(P_4)$ is a
rectangle in order that both $v_2$ and $v_3$ land on edges. Furthermore $n$ must be odd or $f(v_2)$ and
$f(v_3)$ will both be vertices of $P_n$, a situation we already
considered in case (1). Finally, we must have $i+1 < n/2$ in order that the right angles at $f(v_0)$ and
$f(v_1)$ are contained in $P_n$. This situation holds in the first
and second pictures in Table~\ref{tab:Hom47}, with $n=7$ and
respectively $i=1$; $i=2$.

Setting $n=2p+1$, we see that the number of maps of this type is
$$8(2p+1)(p-1) = 4n(n-3) = 4n^2-12n.$$ That all these maps are in fact vertex maps follows from $(\pc_1)$: any affine 1-family must be constant on $v_0$ and $v_1$, but if such a family is not constant on $v_2$ then the images of $v_2$ and $v_3$, when the family parameter varies over $(-1,1)$, must trace out parallel lines - something not possible because these images are confined to \emph{non-parallel} edges of $P_n$.

\medskip In case (3), we may assume $f(v_0) = w_0$ and $f(v_2) = w_i$ for
some $0 < i \leq n/2$. If $i=n/2$, then $n$ is even. Then we
can indeed arrange that $f(v_1)$ and $f(v_3)$ are both on edges, but these
will necessarily be \emph{opposite} (and hence parallel) edges of $P_n$. Without moving
$v_0$ or $v_2$, we can then slide $v_1$ and $v_3$ along these edges in the opposite directions to produce an affine 1-family of maps, contradicting the assumption that $f$ is a vertex map.

On the other hand, if $i<n/2$, and we assume that $f(v_1)$ is on an
edge, let $Q$ be the polygon with vertices $w_0,w_1,\dots,w_i$ and
$Q'$ be the polygon obtained by reflecting $Q$ across the line defined
by $w_0$ and $w_i$. Since $f(v_0)=w_0$, $f(v_2)=w_i$, and $f(w_1)
\in \partial Q$, we conclude that $f(w_3) \in \partial Q'$. But $Q'$
touches the edges of $P$ only at $w_0$ and at $w_i$, so we cannot
arrange that $f(v_3)$ lies on an edge of $Q$. That is, there are no
vertices of $\Hom(P_4,P_m)$ of this type: case (3) is impossible.

Adding up the vertices described by the various cases, we obtain the
total count of rank two vertices of $\Hom(P_4,P_n)$ as claimed in the
proposition.
\end{proof}

\begin{table}
\begin{tabular}{cccc}
\begin{tikzpicture}[scale=1.5]
\draw (1,0) -- (0.6235,0.7818);
\draw (0.6235,0.7818) -- (-0.2225, 0.9749);
\draw (-0.2225, 0.9749) -- (-0.9010, 0.4334);
\draw (-0.9010,0.4334) -- (-0.9010, -0.4334);
\draw (-0.9010, -0.4334) -- (-0.2225,-0.9749);
\draw (-0.2225,-0.9749) -- (0.6235,-0.7818);
\draw (0.6235,-0.7818) -- (1,0);
\draw[lightgray,thick] (-0.9010,0.4334) -- (-0.9010,-0.4334);
\draw[lightgray,thick] (-0.9010,-0.4334) -- (0.7914,-0.4334);
\draw[lightgray,thick] (0.7914,-0.4334) -- (0.7914,0.4334);
\draw[lightgray,thick] (0.7914,0.4334) -- (-0.9010,0.4334);
\end{tikzpicture} &
\hskip0.5cm

\begin{tikzpicture}[scale=1.5]
\draw (1,0) -- (0.6235,0.7818);
\draw (0.6235,0.7818) -- (-0.2225, 0.9749);
\draw (-0.2225, 0.9749) -- (-0.9010, 0.4334);
\draw (-0.9010,0.4334) -- (-0.9010, -0.4334);
\draw (-0.9010, -0.4334) -- (-0.2225,-0.9749);
\draw (-0.2225,-0.9749) -- (0.6235,-0.7818);
\draw (0.6235,-0.7818) -- (1,0);
\draw[lightgray,thick] (0.6235,-0.7818) --  (0.6235,0.7818);
\draw[lightgray,thick] (0.6235,0.7818) -- (-0.4607,0.7818);
\draw[lightgray,thick] (-0.4607,0.7818) -- (-0.4607,-0.7818);
\draw[lightgray,thick] (-0.4607,-0.7818) -- (0.6235,-0.7818);
\end{tikzpicture} & \hskip0.5cm
\begin{tikzpicture}[scale=1.5]
\draw (1,0) -- (0.6235,0.7818);
\draw (0.6235,0.7818) -- (-0.2225, 0.9749);
\draw (-0.2225, 0.9749) -- (-0.9010, 0.4334);
\draw (-0.9010,0.4334) -- (-0.9010, -0.4334);
\draw (-0.9010, -0.4334) -- (-0.2225,-0.9749);
\draw (-0.2225,-0.9749) -- (0.6235,-0.7818);
\draw (0.6235,-0.7818) -- (1,0);
\draw[lightgray,thick] (1,0) -- (0.6235,0.7818);
\draw[lightgray,thick] (0.6235,0.7818) -- (0.2470,0);
\draw[lightgray,thick] (0.2470,0) -- (0.6235,-0.7818);
\draw[lightgray,thick] (0.6235,-0.7818) -- (1,0);
\end{tikzpicture} &
\hskip0.5cm
\begin{tikzpicture}[scale=1.5]
\draw (1,0) -- (0.6235,0.7818);
\draw (0.6235,0.7818) -- (-0.2225, 0.9749);
\draw (-0.2225, 0.9749) -- (-0.9010, 0.4334);
\draw (-0.9010,0.4334) -- (-0.9010, -0.4334);
\draw (-0.9010, -0.4334) -- (-0.2225,-0.9749);
\draw (-0.2225,-0.9749) -- (0.6235,-0.7818);
\draw (0.6235,-0.7818) -- (1,0);
\draw[lightgray,thick] (0.6235,-0.7818) --  (0.6235,0.7818);
\draw[lightgray,thick] (0.6235,0.7818) -- (-0.2225, 0.9749);
\draw[lightgray,thick] (-0.2225, 0.9749) -- (-0.2225,-0.5887);
\draw[lightgray,thick] (-0.2225,-0.5887) -- (0.6235,-0.7818);
\end{tikzpicture}
\end{tabular}
\medskip\caption{Rank two vertices of $\Hom(P_4,P_7)$ up to symmetry}
\label{tab:Hom47}
\end{table}

\begin{table}
\begin{tabular}{cccc}
\begin{tikzpicture}[scale=1.5]
\draw (1,0) -- (0.7071,0.7071);
\draw (0.7071,0.7071) -- (0,1);
\draw (0,1) -- (-0.7071,0.7071);
\draw (-0.7071,0.7071) -- (-1,0);
\draw (-1,0) -- (-0.7071,-0.7071);
\draw (-0.7071,-0.7071) -- (0,-1);
\draw (0,-1) -- (0.7071,-0.7071);
\draw (0.7071,-0.7071) -- (1,0);
\draw[lightgray,thick] (1,0) -- (0,1);
\draw[lightgray,thick] (0,1) -- (-1,0);
\draw[lightgray,thick] (-1,0) -- (0,-1);
\draw[lightgray,thick] (0,-1) -- (1,0);
\end{tikzpicture} &
\hskip0.5cm

\begin{tikzpicture}[scale=1.5]
\draw (1,0) -- (0.7071,0.7071);
\draw (0.7071,0.7071) -- (0,1);
\draw (0,1) -- (-0.7071,0.7071);
\draw (-0.7071,0.7071) -- (-1,0);
\draw (-1,0) -- (-0.7071,-0.7071);
\draw (-0.7071,-0.7071) -- (0,-1);
\draw (0,-1) -- (0.7071,-0.7071);
\draw (0.7071,-0.7071) -- (1,0);
\draw[lightgray,thick] (1,0) -- (0.7071,0.7071);
\draw[lightgray,thick] (0.7071,0.7071) -- (-1,0);
\draw[lightgray,thick] (-1,0) -- (-0.7071,-0.7071);
\draw[lightgray,thick] (-0.7071,-0.7071) -- (1,0);
\end{tikzpicture} &
\hskip0.5cm
\begin{tikzpicture}[scale=1.5]
\draw (1,0) -- (0.7071,0.7071);
\draw (0.7071,0.7071) -- (0,1);
\draw (0,1) -- (-0.7071,0.7071);
\draw (-0.7071,0.7071) -- (-1,0);
\draw (-1,0) -- (-0.7071,-0.7071);
\draw (-0.7071,-0.7071) -- (0,-1);
\draw (0,-1) -- (0.7071,-0.7071);
\draw (0.7071,-0.7071) -- (1,0);
\draw[lightgray,thick] (1,0) -- (0.7071,0.7071);
\draw[lightgray,thick] (0.7071,0.7071) -- (.4142,0);
\draw[lightgray,thick] (.4142,0) -- (0.7071,-0.7071);
\draw[lightgray,thick] (0.7071,-0.7071) -- (1,0);
\end{tikzpicture} &
\hskip0.5cm
\begin{tikzpicture}[scale=1.5]
\draw (1,0) -- (0.7071,0.7071);
\draw (0.7071,0.7071) -- (0,1);
\draw (0,1) -- (-0.7071,0.7071);
\draw (-0.7071,0.7071) -- (-1,0);
\draw (-1,0) -- (-0.7071,-0.7071);
\draw (-0.7071,-0.7071) -- (0,-1);
\draw (0,-1) -- (0.7071,-0.7071);
\draw (0.7071,-0.7071) -- (1,0);
\draw[lightgray,thick] (0.7071,-0.7071) -- (0.7071,0.7071);
\draw[lightgray,thick] (0.7071,0.7071) -- (0,1);
\draw[lightgray,thick] (0,1) -- (0,-0.4142);
\draw[lightgray,thick] (0,-0.4142) -- (0.7071,-0.7071);
\end{tikzpicture}
\end{tabular}
\medskip\caption{Rank two vertices of $\Hom(P_4,P_8)$ up to symmetry}
\label{tab:Hom48}
\end{table}

\subsection{Experimental results}
We end with a table of experimental results for the number of vertices
of $\Hom(P_m,P_n)$ for all $m,n \leq 8$. Our approach to the
computation was as follows. We begin with rational approximations
$Q_m$ and $Q_n$ and compute the polytope $H:=\Hom(Q_m,Q_n)$ with exact
arithmetic, using the
software package \textsf{Polymake}~\cite{GaJoPOLY}. This is possible because of
the explicit facet description of any hom-polytope given by
Proposition~\ref{homproperties}.

\begin{table}\noindent
\begin{minipage}[b]{0.45\linewidth}\centering
\begin{tabular}{||rr|rrrrr|}
m & n & rank 0 & rank 1 & rank 2 & total \\ \hline
3 & 3 & 3 & 18 & 6 & 27 \\
3 & 4 & 4 & 36 & 24 & 64 \\
3 & 5 & 5 & 60 & 60 & 125 \\
3 & 6 & 6 & 90 & 120 & 216 \\
3 & 7 & 7 & 126 & 210 & 343 \\
3 & 8 & 8 & 168 & 336 & 512 \\
4 & 3 & 3 & 12 & 0 & 15 \\
4 & 4 & 4 & 24 & 8 &  36 \\
4 & 5 & 5 & 40 & 80 & 125 \\
4 & 6 & 6 & 60 & 72 & 138 \\
4 & 7 & 7 & 84 & 280 & 371 \\
4 & 8 & 8 & 112 & 240 & 360 \\
5 & 3 & 3 & 30 & 30 & 63 \\
5 & 4 & 4 & 60 & 80 & 144 \\
5 & 5 & 5 & 100 & \emph{60} & \emph{165} \\
5 & 6 & 6 & 150 & \emph{540} & \emph{696} \\
5 & 7 & 7 & 210 & \emph{770} & \emph{987} \\
5 & 8 & 8 & 280 & \emph{1120} & \emph{1408} \\ \hline
\end{tabular}
\end{minipage}
\hspace{0.5cm}
\begin{minipage}[b]{0.45\linewidth}\centering
\begin{tabular}{||rr|rrrrr|}
m & n & rank 0 & rank 1 & rank 2 & total \\ \hline
6 & 3 & 3 & 18 & 12 & 33 \\
6 & 4 & 4 & 36 & 24 & 64 \\
6 & 5 & 5 & 60 & \emph{240} &  \emph{305} \\
6 & 6 & 6 & 90 & \emph{84} & \emph{180} \\
6 & 7 & 7 & 126 & \emph{1008} & \emph{1141} \\
6 & 8 & 8 & 168 & \emph{864} & \emph{1040} \\
7 & 3 & 3 & 42 & 84 & 129 \\
7 & 4 & 4 & 84 & 168 & 256 \\
7 & 5 & 5 & 140 & \emph{770} & \emph{915} \\
7 & 6 & 6 & 210 & \emph{1092} & \emph{1308} \\
7 & 7 & 7 & 294 & \emph{700} & \emph{1001} \\
7 & 8 & 8 & 392 & \emph{2912} & \emph{3312} \\
8 & 3 & 3 & 24 & 48 & 75 \\
8 & 4 & 4 & 48 & 48 & 100 \\
8 & 5 & 5 & 80 & \emph{400} & \emph{485} \\
8 & 6 & 6 & 120 & \emph{288} & \emph{414} \\
8 & 7 & 7 & 168 & \emph{1904} & \emph{2079} \\
8 & 8 & 8 & 224 & \emph{912} & \emph{1144} \\ \hline
\end{tabular}
\end{minipage}
\bigskip\caption{Expected numbers of vertices of $\Hom(P_m,P_n)$}
\label{tab:expected}
\end{table}

However, we do not expect $\Hom(Q_m,Q_n)$ to have the same number of
vertices as $\Hom(P_m,P_n)$. For example, if $m=n \geq 5$, the affine
map that rotates $P_m$ by $\frac{2\pi}{m}$ is a vertex of $\Hom(P_m,P_m)$ at
which $2m$ facets meet; see Example \ref{onmanyfacets}. This map does not exist in
$\Hom(Q_m,Q_m)$. Specifically, the corresponding facets do not all
meet in one point, but various subsets of them do meet to form several
different vertices of $\Hom(Q_m,Q_m)$.
Our problem, then, is to identify the collections of vertices of
$\Hom(Q_m,Q_n)$ that correspond to single vertices of $\Hom(P_m,P_n)$.

Given a polytope $R$ and $\epsilon > 0$, we say that a collection $V$ of
vertices of $R$ is an \emph{$\epsilon$-cluster} if $\|v-w\| <
\epsilon$ for all $v,w \in V$. If all of the vertices of $R$ are
partitioned into a collection of disjoint clusters, we say that
$\epsilon$ successfully partitions $\vertex(R)$.

Note that for any sufficiently large $\epsilon$, we get a single
cluster, and for any sufficiently small $\epsilon$, each vertex forms
a cluster by itself. However, we need intermediate values of
$\epsilon$ that successfully partition $\vertex(\Hom(Q_m,Q_n))$. By
trying several values, we find that for six-digit rational
approximations $Q_m,Q_n$ and for all $m,n \leq 8$, the values
$\epsilon = 10^{-3}$ and $\epsilon = 10^{-4}$ give the same nontrivial
partition. Furthermore, the data resulting from such a partition
agree with what we have proved for the cases $m=3$, $m=4$, $n=3$, and
$n=4$, and also with Proposition~\ref{prop:divisibility}.

Using this partition, we predict the vertex counts shown in the
slanted entries of Table~\ref{tab:expected}; all other values can be obtained from the
theoretical results in this section.



\bibliography{references}
\bibliographystyle{plain}

\end{document}